\documentclass[11pt,reqno]{amsart}

\usepackage[utf8]{inputenc}
\usepackage{amsmath, amssymb, amsthm, mathrsfs}
\usepackage{geometry}
\usepackage[T1]{fontenc}
\usepackage{setspace}

\usepackage{enumitem}
\usepackage{tensor}
\usepackage{graphicx}
\usepackage{subcaption}
\usepackage{xurl}

\usepackage{aliascnt} 
\usepackage{hyperref}
\usepackage[capitalize]{cleveref}

\theoremstyle{plain}
\newtheorem{maintheorem}{Theorem}

\newtheorem*{theorem*}{Theorem}
\newtheorem{theorem}{Theorem}[section]
\newaliascnt{lemma}{theorem}
\newtheorem{lemma}[lemma]{Lemma}
\aliascntresetthe{lemma}
\newaliascnt{proposition}{theorem}
\newtheorem{proposition}[proposition]{Proposition}
\aliascntresetthe{proposition}
\newaliascnt{corollary}{theorem}
\newtheorem{corollary}[corollary]{Corollary}
\aliascntresetthe{corollary}
\theoremstyle{definition}
\newaliascnt{definition}{theorem}
\newtheorem{definition}[definition]{Definition}
\aliascntresetthe{definition}
\newaliascnt{example}{theorem}
\newtheorem{example}[example]{Example}
\aliascntresetthe{example}
\newaliascnt{claim}{theorem}
\newtheorem{claim}[claim]{Claim}
\aliascntresetthe{claim}
\crefname{claim}{Claim}{Claims}
\Crefname{claim}{Claim}{Claims}
\theoremstyle{remark}
\newaliascnt{remark}{theorem}
\newtheorem{remark}[remark]{Remark}
\aliascntresetthe{remark}
\crefname{equation}{eq.}{eqs.}
\Crefname{equation}{eq.}{eqs.}
\newcommand{\R}{\mathbb{R}}
\newcommand{\Z}{\mathbb{Z}}
\newcommand{\N}{\mathbb{N}}
\newcommand{\C}{\mathbb{C}}
\newcommand{\T}{\mathbb{T}}
\newcommand{\tr}{\operatorname{tr}}
\newcommand{\diverg}{\operatorname{div}}
\newcommand{\hess}{\operatorname{Hess}}

\newcommand{\ric}{\operatorname{Ric}}
\newcommand{\scal}{\operatorname{scal}}
\newcommand{\iso}{\operatorname{Iso}}
\newcommand{\fix}{\operatorname{Fix}}
\newcommand{\ima}{\operatorname{Im}}
\newcommand{\conj}{\operatorname{Conj}}
\newcommand{\swap}{\operatorname{Swap}}
\newcommand{\berger}{\mathbb{S}^3_\tau}
\newcommand{\Sphere}{\mathbb{S}}
\newcommand{\campos}{\mathfrak{X}}

\newcommand{\willmore}{\mathcal{W}}
\newcommand{\traceless}{\mathring}
\newcommand{\nablaAmbient}{\overline{\nabla}}
\newcommand{\ip}[2]{\langle #1, #2 \rangle}
\newcommand{\hopftori}{\Sigma^\beta}

\title{Bifurcations of the Clifford Torus as Willmore Surfaces in Berger Spheres}
\author{Caio B. Rodrigues}
\date{}
\address{IMPA - Instituto de Matemática Pura e Aplicada, Rio de Janeiro, RJ, Brazil, 22460-320.}
\email{caio.rodrigues@impa.br}

\begin{document}
 \begin{abstract} 
The Clifford torus in a Berger sphere with parameter $\tau$ is a critical point of the Willmore functional for every $\tau>0$, yielding a smooth path of Willmore surfaces. By estimating the Morse index along this path, we apply bifurcation theory to produce new symmetric Willmore tori emerging from the Clifford torus.
\end{abstract}
 \maketitle
 \section{Introduction}

Let $(M,g)$ be a 3--dimensional Riemannian manifold and $\Sigma \hookrightarrow M$ an isometrically immersed, compact Riemannian surface without boundary. In this work, we are interested in the Willmore functional (energy)  
\begin{equation*}
\willmore(\Sigma) = \int_\Sigma |\traceless{A}|^2 \, d\Sigma,    
\end{equation*}
where $\traceless{A}$ is the traceless second fundamental form of $\Sigma$. In some sense, the simplest surfaces in $M$, after the totally geodesic ones, are the umbilic surfaces (if they exist at all). The functional $\willmore$ measures how far a surface is from being umbilical. In particular, $\Sigma$ is umbilical exactly when $\willmore(\Sigma) = 0$. The infimum of $\willmore$ is zero for any Riemannian ambient manifold, since the Willmore energy of geodesic balls goes to zero as the radius goes to zero. Also, the Willmore functional is invariant under conformal changes of metric \cite{Weiner1978}. In particular, the study of $\willmore$ is equivalent in the space forms, where much is known (e.g. \cite{Bryant1984}, \cite{Pinkall1985}, \cite{Kusner1987}, \cite{Kusner&Cya2024}, and \cite{Riviere&Bernard2014}). 

\medskip
By the Gauss-Bonnet Theorem, the study of $\willmore$ is equivalent to the study of the energy
\begin{equation*}
    \mathscr{W}(\Sigma) = \int_\Sigma H^2 + \overline{K} \, d\Sigma,
\end{equation*}
where $\overline{K}$ is the ambient sectional curvature computed at the tangent planes of $\Sigma$. In the case $M = \R^3$, the functional $\mathscr{W}$ coincides with the classical Willmore energy. In \cite{Willmore1965}, it was proved that the infimum of $\mathscr{W}$ is $4\pi$, and it is attained only by round spheres. He proposed that the ``nicest'' torus in $\R^3$ is the one minimising the functional $\mathscr{W}$ restricted to compact orientable immersed surfaces of genus one. After computations within the class of rotationally symmetric circular tori, one can see that the torus with the least energy within this class is the stereographic projection of the Clifford torus in $\Sphere^3$ (up to scaling). Willmore conjecture that this is the minimizer of what is called today the Willmore energy. The conjecture resisted for almost half a century until it was solved by Marques and Neves using min--max theory \cite{Marques&Neves}. For a survey on this matter, see \cite{Marque&NevesSurvey2014}.

\medskip
The Euler--Lagrange equation for Willmore submanifolds was computed in full generality in \cite{Hu&Li2004}. The formula is given in coordinates with respect to a local orthonormal frame adapted to the submanifold. Nonetheless, several authors have computed the Euler--Lagrange equation for Willmore submanifolds in particular geometric settings from different viewpoints (e.g. \cite{Mondino&Riviere2013}, \cite{wang2025willmoresurfaces4dimensionalconformal}, and \cite{Fabio&Alma2022}). Here, we compute it in the particular case of surfaces, where the formula is expressed in a coordinate--free fashion (see \cref{euler-lagrange equation of willmore energy}). Critical points of $\willmore$ are called Willmore surfaces. 

\medskip
A natural way to proceed with the investigation of the Willmore energy in spaces other than space forms is to move on to the next simplest case, namely, the homogeneous geometries. In this work, we are interested in Berger spheres $\berger \subset \C^2$, $\tau>0$. They are homogeneous deformations of the round sphere $\Sphere^3_1$ obtained by scaling the Hopf circles $\Sphere^1 \to \Sphere^3 \to \Sphere^2\left(\frac{1}{2}\right)$ by $\tau>0$. We properly introduce Berger spheres in \cref{section: first and second variation on Bergers}. For further reading on Berger spheres, see, for example, \cite{Daniel2007}, \cite{Torralbo2012} and \cite{Torralbo&Urbano2012}.

\medskip
Souam and Toubiana showed that there are no umbilical surfaces in $\berger$ (even non--complete ones, see \cite[Th. 1.]{Souam&Tobiana2009}). In particular, the infimum of $\willmore$ is not attained. On the other hand, Carlotto and Mondino showed that there are Willmore $2$-spheres in $\berger$ for $\tau$ sufficiently close to $1$ (the same is true for any analytic metric sufficiently close to a round metric on $\Sphere^3$, see \cite[Th. 1.1.]{Carlotto&Mondino2014}).

\medskip
To the best of our knowledge, all known explicit examples of Willmore surfaces in $\berger$, $\tau \ne 1$, were given in \cite{Manuel1997} (see \cref{corollary: Classification of Hopf Willmore surfaces}), which we briefly describe now. Given a closed curve $\beta$ in $\Sphere^2\left(\frac{1}{2}\right) \subset \C \times \R$, denote by $\hopftori$ the inverse image of $\beta$ by the Hopf projection $\Sphere^3 \to \Sphere^2\left(\frac{1}{2}\right)$. The surface $\hopftori$ is known as a Hopf torus, and the family $\{\hopftori\}_\beta$ as the Hopf tori family. Endow $\hopftori$ with the metric induced from $\berger$. An elastica curve $\beta$ in $\Sphere^2\left(\frac{1}{2}\right)$ is a critical point of the functional 
\begin{equation*}
    \mathcal{F}_\lambda(\beta) = \int_I (k_g^\beta(t))^2 + \lambda \, dt,
\end{equation*}
where $k_g^\beta$ is the geodesic curvature of $\beta$ in $\Sphere^2\left(\frac{1}{2}\right)$ and $\lambda \in \R$ is a constant playing the role of a Lagrange multiplier. Using the principle of symmetric criticality, Barros showed that $\hopftori$ is a Willmore surface of $\berger$ exactly when $\beta$ is an elastica of $\mathcal{F}_\lambda$, with $\lambda = 4\tau^2$. Langer and Singer classified elastica curves in the round $2$-sphere (see \cite{LangerSinger1984} and \cite[Cor. 4.2.]{Langer&Singer1987}). 

\medskip
From the Euler--Lagrange equation of Willmore surfaces in the round sphere $\Sphere^3_1$ it is immediate that minimal surfaces are also Willmore. The previous construction was firstly used by Pinkall, in the case $\tau = 1$, to give the first examples of Willmore surfaces in $\Sphere^3_1$ that are not conformal images of minimal surfaces. Namely, every Hopf torus associated to a non-geodesic elastica of $\mathcal{F}_4$, see \cite{Pinkall1985}.

\medskip
Two Hopf tori associated to elastica curves are of special interest. Let $\Sigma_\tau$, $\tau>0$, denote the Clifford torus 
\begin{equation*}
    \Sphere^1\left(\frac{1}{\sqrt{2}}\right) \times \Sphere^1\left(\frac{1}{\sqrt{2}}\right),
\end{equation*} 
which is the Hopf torus associated to the equator $\Sphere^2\left(\frac{1}{2}\right) \cap (\C \times \{0\})$, endowed with the metric induced from $\berger$. Let $\widetilde{\Sigma}_\tau = \Sigma^{\tilde{\beta}}$, $\tau>\sqrt{2}$, denote the Hopf torus 
\begin{equation*}
    \Sphere^1(a(\tau)) \times \Sphere^1(b(\tau)),
\end{equation*}
where 
\begin{align*}
     a(\tau)  = \sqrt{\frac{1}{2}+c(\tau)}, \quad  b(\tau) = \sqrt{\frac{1}{2}-c(\tau)}, \quad \text{and} \quad  c(\tau) = \frac{1}{2}\sqrt{\frac{\tau^2-2}{\tau^2-1}},
\end{align*} 
are the radii of the circles and the height of the horizontal circle $\Sphere^2\left(\frac{1}{2}\right) \cap (\C \times \{\frac{1}{2}\sqrt{\frac{\tau^2-2}{\tau^2-1}}\})$ parametrized by the elastica $\tilde{\beta} = \tilde{\beta}_\tau$, endowed with the metric induced from $\berger$. These two tori are Willmore surfaces (see \cref{corollary: Classification of Hopf Willmore surfaces,remark: trivial bifurcation at sqrt 2}).

\medskip
Consider the maps 
\begin{equation*}
    \begin{array}{ccc}
  \begin{array}{rcc}
    \swap \colon \Sphere^3 & \to & \Sphere^3 \\
    (z,w) & \mapsto & (w,z)
  \end{array}
  & \quad \text{and} \quad &
  \begin{array}{rcc}
    \rho \colon \Sphere^2\left(\frac{1}{2}\right) & \to & \Sphere^2\left(\frac{1}{2}\right) \\
    (z,c) & \mapsto & (\overline{z},-c)
  \end{array}
\end{array}.
\end{equation*}
They are involutions that define isometries of $\berger$, for every $\tau>0$, and of $\Sphere^2\left(\frac{1}{2}\right)$, respectively. Note that $\swap$ is a fibre-preserving isometry with respect to $\pi$, covering $\rho$. Since $\rho$ sends the horizontal circle at height $c(\tau)$ to the one at height $-c(\tau)$ and both of them approach the equator as $\tau \searrow \sqrt{2}$, both $\widetilde{\Sigma}_\tau$ and $\swap(\widetilde{\Sigma}_\tau)$ approach $\Sigma_\tau$ as $\tau \searrow \sqrt{2}$.

\medskip
This can be interpreted as a pitchfork bifurcation of the equator as an elastica in $\Sphere^2\left(\frac{1}{2}\right)$, which lifts to a pitchfork bifurcation of the Clifford torus as a Willmore surface in $\berger$. Moreover, the torus $\widetilde{\Sigma}_\tau$ also bifurcates into paths of symmetric Willmore tori. For this reason, we refer to $\widetilde{\Sigma}_\tau$ as the \emph{bifurcated brother}.

\medskip
On the one hand, it follows from Langer and Singer's stability analysis of elastica curves in the round $\Sphere^2$ that, when we restrict the functional $\willmore$ to the Hopf tori family, the Clifford torus $\Sigma_\tau$ is stable exactly when $0 < \tau \leq \sqrt{2}$, and $\widetilde{\Sigma}_\tau$ is stable for every $\tau>\sqrt{2}$. On the other hand, we computed the second variation of the Willmore energy (see \cref{second variation of surfaces in the Berger sphere}) and studied the index of $\Sigma_\tau$ and $\widetilde{\Sigma}_\tau$ as $\tau$ varies. Let $i_\tau(\Sigma)$ and $n_\tau(\Sigma)$ denote the index and the nullity of a Willmore surface $\Sigma \hookrightarrow \berger$, and let $\left\lfloor \cdot \right\rfloor$ denote the floor function.
\begin{maintheorem}[Stability of the Clifford torus]\label{theorem: Stability of the Clifford torus}
The Clifford torus \(\Sigma_\tau\) is stable as a Willmore surface in a Berger sphere if and only if \(0 < \tau \le 1\). Its index and nullity satisfy
\begin{equation*}
    i_\tau(\Sigma_\tau) \ge \left\lfloor \tau^2 \right\rfloor \text{ if }  \tau >1,
\quad \text{and} \quad n_\tau(\Sigma_\tau) = 
\begin{cases}
2, & \text{if } \tau \in (0,1) \cup \big((1,\infty) \setminus \mathcal{D}\big), \\
8, & \text{if } \tau = 1, \\
3, & \text{if } \tau= \sqrt{2},
\end{cases}
\end{equation*}
where $\mathcal{D}$ is an infinite countable subset of $(1,\infty) \setminus \{\sqrt{2}\} $. Moreover, $n_\tau(\Sigma_\tau)$ is an even integer and $n_\tau(\Sigma_\tau) \ge 4$ whenever $\tau \in \mathcal{D}$. 
\end{maintheorem}
The set $\mathcal{D}$ is explicitly given by the roots of a countable set of polynomials of degree 6. These roots are precisely the zeros of the eigenvalues (viewed as functions of $\tau$) of the second variation operator $\mathcal{L}_{\Sigma_\tau}$ associated with $\Sigma_\tau$ as a Willmore surface.  
\begin{maintheorem}[Stability of the bifurcated brother]\label{theorem: instabilidade do irmao paralelo}
    The Willmore surface $\widetilde{\Sigma}_\tau$ is unstable for every $\tau>\sqrt{2}$. Moreover, its index tends to infinity as $\tau \nearrow +\infty$.
\end{maintheorem}

In light of the above theorems, we may interpret the maps $\tau \mapsto \Sigma_\tau$ and $\tau \mapsto \widetilde{\Sigma}_\tau$ as paths of critical points of the Willmore functional along which the index changes. Bifurcation theory tells us to expect the existence of new critical points, that is, Willmore surfaces with the same topology emerging from $\Sigma_\tau$ and $\widetilde{\Sigma}_\tau$. Indeed, by applying the bifurcation from a simple eigenvalue theorem of Crandall--Rabinowitz \cite{CrandallRabinowitz1971}, together with the Maximal Torus Theorem, we prove bifurcation results, which we state below in a rather imprecise manner. See \cref{theorem: Bifurcations of the Clifford Torus,theorem: Bifurcations of the bifurcated brother} for the precise statements. (For more recent applications of Bifurcation Theory to Geometric Analysis see the non--exhaustive list: \cite{Bettiol&Piccione2016}, \cite{Bettiol&Piccione2025-2}, \cite{Bettiol&Piccione2025}, \cite{Lima&Lira2014}, \cite{Luis&Piccione2013}; for surveys, see \cite{Bettiol&Piccione2020}, \cite{Bettiol&Piccione2022}).

\medskip
Given $\mu, \nu \in \R$, define the maps 
\begin{equation*}
    R_{\mu,\nu}, \, \conj \colon \Sphere^3 \to \Sphere^3, \quad (z,w) \mapsto (e^{i\mu}z, e^{i\nu}w), \, (\overline{z},\overline{w}).
\end{equation*} 
These maps are isometries of $\berger$ for every $\tau>0$. We set 
\begin{equation*}
    K = \operatorname{span}\{R_{\theta,-\theta}, \conj : \theta \in \R\} \subset \iso(\berger).
\end{equation*}
\begin{maintheorem}[Bifurcations of the Clifford torus]
   There exists a sequence of Berger parameters $\{\tau_m\}_{m \in \N_0}$ and a tubular neighbourhood $V \subset \Sphere^3$ of the Clifford torus such that, for each $\tau_m$, there exists a smooth path of embedded Willmore surfaces $(-\delta_m,\delta_m) \ni t \mapsto X_{\tau_m(t)}^t \subset \Sphere^3_{\tau_m(t)}$, with $X_{\tau_m(t)}^t \subset V$, which are graphical over the Clifford torus, and such that 
   \begin{enumerate}
      \item[$(a)$] $\tau_0 = \sqrt{2}$, $\tau_1 = 1$, $\{\tau_m\}_{m \in \N}$ is strictly increasing for $m \geq1$, and $\tau_m \nearrow +\infty$.
      \vspace{0.2cm}
      \item[$(b)$] The surface $X_{\tau_m(0)}^0$ is the Clifford torus $\Sigma_{\tau_m}$, and $X_{\tau_m(t)}^t \neq \Sigma_{\tau_m}$ when $t \neq 0$.
      \vspace{0.2cm}
      \item[$(c)$] The group $K$ is a subgroup of the symmetries of $X_{\tau_m(t)}^t$.
      \vspace{0.2cm}
      \item[$(d)$] If $X \subset \Sphere^3_{\tau_m(t)}$ is a Willmore surface, with $X \subset V$, which is graphical over $\Sigma_{\tau_m}$, then $X=\Sigma_{\tau_m(t)}$ or $X=X_{\tau_m(s)}^s$ for some $s$ with $\tau_m(s) = \tau_m(t)$.
      \vspace{0.2cm}
      \item [$(e)$] The bifurcation at $\tau_0$ is a reparametrization of the bifurcated brother $\tau \mapsto \widetilde{\Sigma}_\tau$.
      \vspace{0.2cm}
      \item[$(f)$] The bifurcation at $\tau_1$ is given by a one-parameter family of tori in $\Sphere^3_1$ that are congruent\footnote{Two surfaces are said to be congruent when there exists an ambient isometry taking one to the other.} to the Clifford torus.
      \vspace{0.2cm}
      \item[$(g)$] For $m \geq 2$, the Willmore surface $X_{\tau_m(t)}^t$ is not congruent to any Hopf torus when $t \neq 0$.
   \end{enumerate}
\end{maintheorem}
We explicitly know the variational vector field of the variation $t \mapsto X_{\tau_m(t)}^t$ of $\Sigma_{\tau_m}$. It is given by an eigenfunction of the second variation operator $\mathcal{L}_{\Sigma_{\tau_m}}$ associated with $\Sigma_{\tau_m}$ as a Willmore surface (see \cref{Clifford torus second variation operator}). The bifurcation parameter $\tau_m$ is a root of the corresponding eigenvalue, as a function of $\tau$, for which we have a closed formula (see \cref{proposition: spectrum of the Clifford torus}). We also verify that there is no other bifurcation parameters. Similar remarks apply to the following bifurcation result for $\widetilde{\Sigma}_\tau$:
\begin{maintheorem}[Bifurcations of the bifurcated brother]
   For each $m \in \{2,3\}$, there exists a Berger parameter $\tilde{\tau}_m$, a tubular neighbourhood $\widetilde{V}_m \subset \Sphere^3$ of the bifurcated brother $\widetilde{\Sigma}_{\tilde{\tau}_m}$, and a smooth path of embedded Willmore surfaces $(-\tilde{\delta}_m, \tilde{\delta}_m) \ni t \mapsto X_{\tilde{\tau}_m(t)}^t \subset \Sphere^3_{\tilde{\tau}_m(t)}$, with $X_{\tilde{\tau}_m(t)}^t \subset \widetilde{V}_m$, which are graphical over $\widetilde{\Sigma}_{\tilde{\tau}_m}$, and such that
   \begin{enumerate}
      \item[$(a)$] The surface $X_{\tilde{\tau}_m(0)}^0$ is the bifurcated brother $\widetilde{\Sigma}_{\tilde{\tau}_m}$, and $X_{\tilde{\tau}_m(t)}^t \neq \widetilde{\Sigma}_{\tilde{\tau}_m}$ when $t \neq 0$.
      \vspace{0.2cm}
      \item[$(b)$] The group $K$ is a subgroup of the symmetries of $X_{\tilde{\tau}_m(t)}^t$.
      \vspace{0.2cm}
      \item[$(c)$] If $X \subset \Sphere^3_{\tilde{\tau}_m(t)}$ is a Willmore surface, with $X \subset \widetilde{V}_m$, which is graphical over $\widetilde{\Sigma}_{\tilde{\tau}_m}$, then $X=\widetilde{\Sigma}_{\tilde{\tau}_m(t)}$ or $X=X_{\tilde{\tau}_m(s)}^s$ for some $s$ with $\tilde{\tau}_m(s)=\tilde{\tau}_m(t)$.
      \vspace{0.2cm}
      \item[$(d)$] The Willmore surface $X_{\tilde{\tau}_m(t)}^t$ is not congruent to any Hopf torus when $t \neq 0$.
   \end{enumerate}
\end{maintheorem}

Although the index of $\widetilde{\Sigma}_\tau$ grows arbitrarily with $\tau$, exactly two bifurcations occur in the above theorem. This is because the variations of $\widetilde{\Sigma}_\tau$ showing that the index diverges do not occur within the family of $K$--invariant tori.

\medskip
This paper is organized as follows. In \cref{section: conventions and notations}, we fix some notations and conventions. In \cref{section: First Variation}, we compute the first variation of the Willmore functional for surfaces. In \cref{section: first and second variation on Bergers}, particularizing the study of the Willmore energy to Berger spheres, we present a new explicit example of a Willmore 2--sphere (\cref{ex: equador}), prove a non--existence result of CMC Willmore surfaces (\cref{Prop: Non-existence of CMC Willmore surfaces}), and compute the second variation for Willmore surfaces. In \cref{section: hopf tori}, we study the family of Hopf tori in a Berger sphere. In particular, we compute the spectrum of the second variation operators for the Clifford torus and the bifurcated brother as Willmore surfaces. In \cref{section: Bifurcations of the Clifford torus}, we establish bifurcation results. In \cref{section: Final Remarks}, we present our final remarks. In the Appendices \ref{app:aux} and \ref{app: lemma's proofs} we state the Crandall--Rabinowitz bifurcation from a simple eigenvalue theorem and prove auxiliary lemmas, respectively.

\medskip
\textbf{Acknowledgements}. I would like to express my deep gratitude to Professor Lucas Ambrozio for his guidance during the preparation of this work, which is part of my ongoing PhD thesis. I am grateful to Professor Renato Bettiol for introducing me to bifurcation theory and for sharing insightful ideas. I also thank my colleague Ivan Miranda for his interest in this work and for stimulating conversations. This research was supported by CNPq -- Conselho Nacional de Desenvolvimento Científico e Tecnológico.

 \section{Conventions and notations}\label{section: conventions and notations}

Let $(M^3,g)$ be an ambient Riemannian manifold, $\Sigma \hookrightarrow M$ be an isometrically immersed Riemannian surface, with Riemannian metric also denoted by $g$, and assume there exists $N \in \campos^\perp(\Sigma)$ a unit normal vector field along $\Sigma$. Let $\nablaAmbient, \overline{R}, \overline{\hess}$, ... correspond to $M$ and $\nabla, R, \hess$, ... correspond to $\Sigma$. From now on, we adopt the following conventions.

\medskip
The Riemann curvature tensor is defined by
\begin{align*}
\overline{R}(X,Y)Z &= \nablaAmbient_X \nablaAmbient_Y Z - \nablaAmbient_Y \nablaAmbient_X Z - \nablaAmbient_{[X,Y]} Z, \\
\overline{R}(X,Y,W,Z) &= g(\overline{R}(X,Y)W,Z).
\end{align*}
The Ricci tensor is defined by
\begin{equation*}
    \overline{\ric}(X,Y) = \tr \overline{R}(X, \cdot,\cdot,Y).
\end{equation*}
The scalar curvature is 
\begin{equation*}
    \overline{\scal} = \tr \overline{\ric}.
\end{equation*}
The second fundamental form of $\Sigma$ is given by
\begin{align*}
S(X) &= \nablaAmbient_X N, \\
A(X,Y) &= \langle \nablaAmbient_X N, Y \rangle.
\end{align*}
The mean curvature is defined as
\begin{equation*}
H = \tr A. 
\end{equation*}
Let $f$ be a function on $M$. The Laplace–Beltrami operator is
\begin{align*}
  \overline{\Delta} f = \overline{\diverg} \ \overline{\nabla} f = \tr ( X \mapsto \overline{\nabla}_X \overline{\nabla} f).
\end{align*}
The Jacobi operator associated to $\Sigma$ is given by
\begin{equation*}
    J_\Sigma = \Delta_\Sigma + \left( |A|^2 + \ric_M(N,N) \right).
\end{equation*}
When necessary, we use subscripts to indicate the manifold to which the operator, tensor, etc belongs.
 \section{First variation of the Willmore energy}\label{section: First Variation}
Let $F_0\colon \Sigma^2 \hookrightarrow M^3$ be a compact two-sided isometrically immersed Riemannian surface and denote by $g$ the Riemannian metric of $M$. We define the Willmore energy of $\Sigma$ \footnote{We actually mean the Willmore energy of $F_0$ and should write $\willmore(F_0)$ to be precise.} as
\begin{equation*}
    \willmore(\Sigma) = \int_\Sigma |\traceless{A}|^2 \, d\Sigma,
\end{equation*}
where $\traceless{A} = A - \frac{1}{2} H g$ is the traceless second fundamental form, and $d\Sigma$ denotes the Riemannian volume element of $\Sigma$. Let $F \colon \Sigma \times I \to M$ be a smooth variation of $F_0$, that is, $F_0 = F(\cdot,0)$ and each $F_t = F(\cdot,t)$ is an immersion. Let $\Sigma_t$ be $\Sigma$ equipped with the metric induced by $F_t$, which we also denote by g. Let $x^1, \ x^2$ be local coordinates on $\Sigma$ and denote
\begin{equation*}
    \partial_i = \frac{\partial F}{\partial x^i}, \quad \partial_t = \frac{\partial F}{\partial t}.
\end{equation*}
Let $N_t$ be a unit normal vector field along $F_t$. Suppose that the variational vector field is normal at for all $t$, that is,
\begin{equation*}
    \partial_t = f_t N_t,
\end{equation*}
where $f_t \colon \Sigma \to \R$ is an arbitrary smooth function. We use the simplified notations $\Sigma = \Sigma_t$, $N=N_t$, $f=f_t$, $g=g_t$, $A=A_t$, $\dots$ Recall that
\begin{align}
\partial_t g_{ij} &= 2fA_{ij}, \notag \\
\partial_t g^{ij} &= -2f g^{ip} g^{jq} A_{pq}, \label{variation of inverse of the metric} \\
\partial_t\sqrt{\det g} &= \diverg(\partial_t)\sqrt{\det g} =fH\sqrt{\det g}, \label{variation of sqrt det g} \\
\partial_t \ d\Sigma &= \diverg(\partial_t) \ d\Sigma = fH \ d\Sigma, \label{variation of volume element}
\end{align} 
Also,
\begin{equation}
    |\traceless{A}|^2 = |A|^2-\frac{1}{2}H^2. \label{Squared Norm of Traceless secondff}
\end{equation}
Define the function  
\begin{equation*}
    \mathcal{E}_\willmore(\Sigma) = -\Delta H - |\traceless{A}|^2 H + 3\overline{\ric}(N,N) H -2g(\overline{R}_N + \overline{\ric}, A) - g(\nablaAmbient \ \overline{\scal}, N)+ 2 (\nablaAmbient_N \overline{\ric})(N,N),
\end{equation*}
where $\overline{R}_N = \overline{R}(N,\cdot,\cdot,N)$. 
\begin{theorem} \label{euler-lagrange equation of willmore energy}
The first derivative of the Willmore energy under the variation F is given by 
  \begin{equation*}
      \left.\frac{d}{dt}\right|_{t=0} \willmore(\Sigma_t) = \int_\Sigma \mathcal{E}_{\willmore}(\Sigma)f \, d\Sigma.
  \end{equation*}
  In particular, the surface $\Sigma$ is Willmore if and only if the Euler--Lagrange equation
  \begin{align*}
     \mathcal{E}_\willmore(\Sigma) = 0
 \end{align*}
 is satisfied.
\end{theorem}
Sometimes it is convenient to regard $\mathcal{E}_\willmore(\Sigma)$ as the first variation operator of the Willmore energy of $\Sigma$, acting on functions $\Sigma \to \R$ via multiplication.
\begin{proof}
From \cref{variation of volume element,Squared Norm of Traceless secondff} we obtain that
\begin{align}
\frac{d}{dt} \willmore(\Sigma_t)
&= \int_\Sigma \partial_t \left( |\traceless{A}|^2 \, d\Sigma \right) \notag \\
&= \int_\Sigma \Big(\partial_t \big( |A|^2 - \frac{1}{2}H^2 \big) \, d\Sigma 
    + |\traceless{A}|^2 \, \partial_t(d\Sigma) \Big) \notag \\
&= \int_\Sigma (\partial_t |A|^2 - H \partial_t H ) d\Sigma 
    + |\traceless{A}|^2 f H \, d\Sigma, \label{variation of willmore energy incomplete 1}
\end{align}
In coordinates, $|A|^2=A_{ij}A_{kl}g^{ik}g^{jl}$. To compute $\partial_t|A|^2$ and $\partial_tH$, notice that 
\begin{align}
    \partial_t A_{ij} &= g(\nablaAmbient_{\partial_t} \nablaAmbient_{\partial_i} N,\partial_j) + g(\nablaAmbient_{\partial_i} N,\nablaAmbient_{\partial_t} \partial_j) \notag \\
    &= g(\overline{R}(\partial_t,\partial_i)N + \nablaAmbient_{\partial_i} \nablaAmbient_{\partial_t} N, \partial_j) + g(\nablaAmbient_{\partial_i} N,\nablaAmbient_{\partial_j} \partial_t) \notag \\
    &= g(\nablaAmbient_{\partial_i} \nablaAmbient_{\partial_t} N, \partial_j) - f \overline{R}(N,\partial_i,\partial_j,N) + f g(S(\partial_i),  S(\partial_j)). \label{variation of A_ij incomplete}
\end{align}
We claim that 
\begin{align}
    \nablaAmbient_{\partial_t}N= - \nabla f. \label{grad f}
\end{align} 
Indeed, since the variation is normal for all $t$,
\begin{align*}
    g(\nablaAmbient_{\partial_t} N, \partial_i)=-g(N,\nablaAmbient_{\partial_i} \partial_t)=-\partial_i g(N,\partial_t)=g(-\nabla f,\partial_i). 
\end{align*}
Define $\overline{R}_N=\overline{R}(N,\cdot,\cdot,N)$. By \cref{variation of A_ij incomplete,grad f},
\begin{align}
    \partial_t A_{ij} &= -(\hess f)_{ij} - f(\overline{R}_N)_{ij} + f(A^2)_{ij}. \label{variation of A_ij complete}
\end{align}
By \cref{variation of inverse of the metric,variation of A_ij complete},
\begin{align}
    \partial_t|A|^2 &= 2\partial_t A_{ij}A_{kl}g^{ik}g^{jl}+2A_{ij}A_{kl} \partial_t g^{ik}g^{jl} \notag \\ 
    &= 2(-g(\hess f,A)-f g(\overline{R}_N,A) + f \, \tr(A^3)) - 4f \, \tr(A^3) \notag \\
    &= -2(g(\hess f,A)+ f g(\overline{R}_N,A) + f \, \tr(A^3)), \label{variation of squared 2ff}
\end{align}
and 
\begin{equation}
    \partial_t H = - J_{\Sigma}(f), \label{variation of H} 
\end{equation}
where $J_{\Sigma} = \Delta_{\Sigma} + \left( |A|^2 + \overline{\ric}(N,N) \right)$ is the Jacobi operator of $\Sigma$.

\medskip
Observe that if $(L,\ip{\cdot}{\cdot})$ is a compact Riemannian manifold, $h \in C^\infty(L)$ is a smooth function, and $B$ is a $2$--tensor, then the following integration by parts holds:
\begin{equation*}
    \int_L \ip{\hess h}{B} dL = \int_L h \diverg^2(B) dL,
\end{equation*}
where $\diverg^2 = \diverg \circ \diverg$.
By eqs. (\ref{variation of willmore energy incomplete 1}), (\ref{variation of squared 2ff}), (\ref{variation of H}), and the above observation,
\begin{align*}
    \frac{d}{dt} \willmore(\Sigma_t)
&= \int_\Sigma f\Big( -2\big( g(\overline{R}_N,A)+\tr(A^3) \big) + |\traceless A|^2 H \Big) - 2 g(\hess f,A) + H J_{\Sigma}(f) \, d\Sigma \\
&= \int_{\Sigma} f\Big(J_{\Sigma}(H) - 2 \diverg^2(A) - 2\big(g(\overline{R}_N,A) + \tr(A^3) \big) + |\traceless A|^2 H \Big) \, d\Sigma.
\end{align*}
Provisionally, denote the first variation operator by $\mathcal{I}_\willmore(\Sigma)$. From the above equation, we obtain
\begin{equation*}
   \mathcal{I}_\willmore(\Sigma)= J_{\Sigma}(H) - 2 \diverg^2(A) - 2\big(g(\overline{R}_N,A) + \tr(A^3) \big) + |\traceless A|^2 H.
\end{equation*}
By the algebraic identity $2\tr(A^3)=H(|A|^2+2|\traceless A|^2)$ (which comes from the Cayley-Hamilton Theorem), we have that
\begin{align}
    \mathcal{I}_\willmore(\Sigma)=\Delta H + \overline{\ric}(N,N)H - 2\big( \diverg^2(A) + g(\overline{R}_N,A)\big) - |\traceless A|^2 H \label{euler-lagrange os surface incomplete 1}.
\end{align}
From the Gauss-Codazzi equation, for any $X \in \campos(\Sigma),$
\begin{align*}
    (\diverg A) (X) &= g^{ij}(\nabla_{\partial_i}A)(X,\partial_j)\\
    &=g^{ij}\big(  (\nabla_X A)(\partial_i,\partial_j) - \overline{R}(\partial_i,X,\partial_j,N)  \big) \\
    &= \tr(\nabla_X A) + \tr \overline{R}(X, \cdot, \cdot, N) \\
    &=dH(X) + T(X),
\end{align*}
where $T = \overline{\ric}(N,\cdot)$. Hence,
\begin{align}
    \diverg^2(A) = \Delta H + \diverg T. \label{square divergence of A 1}
\end{align}
Using the contracted second Bianchi identity,
\begin{align}
    \diverg(T) &= g^{ij}(\nabla_{\partial_i} T)(\partial_j) \notag \\
               &= g^{ij}\big(\partial_i(\overline{\ric}(N,\partial_j))-\overline{\ric}(N,\nabla_{\partial_i} \partial_j)  \big) \notag \\
               &= g^{ij} \big((\nablaAmbient_{\partial_i} \overline{\ric} )(N,\partial_j) + \overline{\ric}(\nablaAmbient_{\partial_i} N, \partial_j) + \overline{\ric}(N, \nablaAmbient_{\partial_i} \partial_j - \nabla_{\partial_i} \partial_j)  \big) \notag \\
               &= \diverg(\overline{\ric})(N) - (\nablaAmbient_N \overline{\ric})(N,N) + g(\overline{\ric}, A)- H \overline{\ric}(N,N) \notag \\
               &= \frac{1}{2}g(\nablaAmbient scal, N) - (\nablaAmbient_N \overline{\ric})(N,N) + g(\overline{\ric}, A)- H \overline{\ric}(N,N).
               \label{divergence of T}
\end{align}
Also, observe that 
\begin{align}
    g(\overline{\ric} + \overline{R}_N,A) = g(\overline{\ric} + \overline{R}_N,\traceless A) + \frac{H}{2} \, \overline{scal}. \label{algebraic identity}
\end{align}
The result follows from \cref{euler-lagrange os surface incomplete 1,square divergence of A 1,divergence of T,algebraic identity}.

\end{proof}
 \section{First and second variation in Berger spheres}\label{section: first and second variation on Bergers}

Let $\Sphere^n(r)$ denote the $n$-dimensional round sphere of radius $r$, centred at the origin, and denote $\Sphere^n = \Sphere^n(1)$. A \emph{Berger sphere} $\berger$ is a 3-sphere $\Sphere^3 \subset \C^2$ equipped with the Riemannian metric $g_\tau$ defined by
\begin{equation}\label{eq: metrica de berger}
   g_\tau(X,Y) = \ip{X}{Y} + (\tau^2 - 1)\ip{X}{V} \ip{Y}{V},    
\end{equation}

where $\ip{\cdot}{\cdot}$ is the canonical round metric on $\Sphere^3$, $V(z,w) = (iz, iw)$, and $\tau>0$. The Levi-Civita connection of $\berger$ is given by
\begin{equation}\label{eq: conexão da Esfera de Berger}
\overline{\nabla}_X Y = \nabla^{can}_X Y + (\tau^2 - 1)\left(\ip{X}{V}(iY)^\top + \ip{Y}{V}(iX)^\top\right),
\end{equation}
where $\nabla^{can}$ is the Levi-Civita connection of $(\Sphere^3, \ip{\cdot}{\cdot})$ and $(\cdot)^\top$ denotes the orthogonal projection onto the tangent bundle $T\Sphere^3$. The \emph{Hopf fibration} $\pi \colon \Sphere^3 \to \Sphere^2\left(\frac{1}{2}\right)$ is defined by
\begin{equation}\label{eq: Hopf submersion}
    \pi(z,w) = \left(z\overline{w}, \frac{1}{2}(|z|^2 - |w|^2)\right).
\end{equation}
\noindent
Henceforth, we will denote $g=g_\tau$. The Riemann tensor and the Ricci tensor of $\berger$ are given by
\begin{align}\label{Riemann tensor of the Berger sphere}
    \overline{R}(X,Y)Z 
    &= (4-3\tau^2)(g(Y,Z)X - g(X,Z)Y) \notag \\
    & +4(1-\tau^2)(g(Z,\xi)g(X,\xi)Y - g(Y,\xi)X) \\
    & +4(1-\tau^2)(g(X,Z)g(Y,\xi) - g(Y,Z)g(X,\xi))\xi, \notag 
\end{align}
\begin{align}\label{Ricci tensor of the berger Sphere}
    \overline{\ric}(X,Y)= 2(2-\tau^2)g(X,Y) - 4(1-\tau^2)g(X,\xi) g(Y,\xi), 
\end{align}
where $\xi = \frac{1}{\tau} V$ is a unit Killing vector field of $\berger$, tangent to Hopf fibres. It has the property that, for any $X \in \campos(\berger)$,
\begin{align}
    \nablaAmbient_X \xi = \tau X \wedge \xi, \label{xi property}
\end{align}
where $\wedge$ is a vector product defined in the following way: let $dV_\tau$ be the volume element of $\berger$. For any $X,Y \in \campos(\berger)$, the vector field $X \wedge Y$ is defined by the property
\begin{equation*}
    \forall Z \in \campos(\berger), \quad g_\tau(X \wedge Y, Z)=dV_\tau(X,Y,Z).
\end{equation*}
When $\tau\ne 1$, the isometry group of $\berger$ is given by 
\begin{equation*}
    \iso(\berger) = U^+ \cup U^-,
\end{equation*} 
where $U^+$ is the unitary group $U(2)$, and $U^-=\{ \conj \circ \ \psi : \psi \in U^+\}$ with $\conj(z,w) = (\overline{z},\overline{w})$ denoting the complex conjugation map. Note that $\conj$ reverses the orientation of both the Hopf fibres and the horizontal distribution defined by $\pi$. Consequently, all isometries of $\berger$ preserves orientation whenever $\tau\ne 1$.

\medskip
Let $F\colon \Sigma \hookrightarrow \berger$ be an isometrically immersed oriented surface and $N$ be a positive unit normal vector field along $\Sigma$ and $A$ the second fundamental form corresponding to $N$. The function
\begin{equation*}
    C = g(\xi,N),
\end{equation*}
naturally attached to $\Sigma$, that measures the cosine of the angle between $\xi$ and $N$, is called the \emph{angle function}. Let 
\begin{equation*}
    T = \xi-CN,
\end{equation*}
be the tangential component of $\xi$ in $T\Sigma$. Denote by $J$ the positively oriented rotation by $\frac{\pi}{2}$ on $T\Sigma$, that is, for any $X \in \campos(\Sigma)$, 
\begin{equation*}
    J(X) = N \wedge X.
\end{equation*}
We refer the reader to \cite{Daniel2007} for these and further geometric properties of Berger spheres and 3-homogeneous geometries.

\medskip
From \cref{euler-lagrange equation of willmore energy}, we recover the Euler--Lagrange equation for Willmore surfaces in $\berger$, as in \cite{Fabio&Alma2022} (see Proposition 2).
\begin{corollary}\label{euler-lagrange equation of willmore surfaces in the Berger sphere}
    Let $\Sigma \hookrightarrow \berger$ be an oriented compact isometric immersed surface. Then, the first variation operator of the Willmore energy is
    \begin{equation*}
        \mathcal{E}_\willmore(\Sigma) = -\Delta H - \big(|\traceless A|^2+4(1-\tau^2)(1+C^2)\big)H+16(1-\tau^2)A(T,T).
    \end{equation*}
\end{corollary}

\begin{example}\label{ex: equador} 
In \cite[Th. 1.1.]{Carlotto&Mondino2014}, the existence of Willmore 2--spheres was established, by perturbative arguments, for any analytic metric sufficiently close to the round metric in $\Sphere^3$. Here we present an explicit example in $\berger$. Namely, the equator 
$$
\Sphere^2 = \{(a,w) \in \Sphere^3 \colon a \in \R, w \in \C, \ a^2+|w|^2 = 1 \}
$$
endowed with the metric induced by the ambient space $\berger$, which we denote by $\Sphere^2_\tau$. Consider the parametrization 
\begin{equation*}
    (-1,1) \times (0,\pi) \to \Sphere^2, \quad (a,\theta) \mapsto (a,\sqrt{1-a^2}e^{i\theta}).
\end{equation*} 
In these coordinates, the tangent component of $\xi$ is proportional to 
\begin{equation*}
    \partial_\theta (a,\theta)=(0,i\sqrt{1-a^2}e^{i\theta}).
\end{equation*}
By \cref{eq: conexão da Esfera de Berger}, we obtain $\nablaAmbient_{\partial_\theta}\partial_\theta = 0$, and hence $A(T,T) = 0$. Since the equator $\Sphere^2_\tau$ is a minimal surface of $\berger$ (see \cite[Th. 1.]{Torralbo&Urbano2010}), it follows from \cref{euler-lagrange equation of willmore surfaces in the Berger sphere} that it is also a Willmore one. Furthermore, note that as the action $\iso(\berger) \curvearrowright \Sphere^3$ is transitive, not only the above equator, but all great spheres of the round $\Sphere^3$, when endowed with the metric induced by $\berger$, are Willmore surfaces.
\end{example}

In light of \cref{euler-lagrange equation of willmore surfaces in the Berger sphere}, it is natural to look among CMC (constant-mean-curvature) surfaces to see which of them are also Willmore surfaces. Besides the example above, the aforementioned bifurcated brother $\widetilde{\Sigma}_\tau$, for $\tau \in (\sqrt{2},\infty)$, is also an example of such surfaces in $\berger$, which is non--minimal. On the other hand, we have the following non--existence result. 

\begin{proposition}\label{Prop: Non-existence of CMC Willmore surfaces}
    Let $0<\tau<1$ and $\Sigma \hookrightarrow \berger$ be an isometrically immersed oriented compact surface. If $\Sigma$ is a constant mean curvature Willmore surface, then $\Sigma$ is minimal, $T$ is an asymptotic direction, and the angle function $C$ is constant along the flow lines of $T$.
\end{proposition}
\begin{proof}
    From \cref{xi property},
    \begin{align*}
        A(T,T)&=-g(N, \nablaAmbient_T T)= -g(N,\nablaAmbient_T \xi) +g(N,\nablaAmbient_T(CN)) \\
              &=-\tau g(N,T \wedge \xi) + dC(T) = dC(T).
    \end{align*}
Since $\Sigma$ is both CMC and Willmore, by \cref{euler-lagrange equation of willmore surfaces in the Berger sphere}, we get
\begin{align*}
  \big(|\traceless{A}|^2 + 4(1-\tau^2)(1+C^2)\big)H=16(1-\tau^2)dC(T).    
\end{align*}
Suppose, by contradiction, that $\Sigma$ is not minimal. Changing the orientation of $\Sigma$, we may assume that $H>0$. From the assumption $0<\tau<1$ and the above equation, we obtain
\begin{align*}
    dC(T)>0.
\end{align*}
For each $p\in \Sigma$, let $\gamma_p: \R \to \Sigma$ be the integral curve of $T$ with $\gamma_p(0)=p$. Define $f_p = C \circ \gamma_p$. From the above inequality, it follows that
\begin{align*}
     f_p' > 0, \quad \text{for every $p \in \Sigma$}.
\end{align*}
Since $C$ is bounded, $f_p$ is a bounded increasing function and by the mean value theorem there exists a sequence $t_n \nearrow +\infty$ such that $f_p'(t_n)\to 0$. Since $\Sigma$ is compact, we may suppose that $z_n :=\gamma_p(t_n) \to z$, for some $z \in \Sigma$. In this case,
\begin{align*}
    0 < f_z'(0) = dC(\xi(z)) = \lim_{n \to \infty} dC(\xi(z_n)) = \lim_{n \to \infty} f_p'(t_n)= 0. 
\end{align*}
A contradiction. Thus, $\Sigma$ is a minimal surface.

\end{proof}

We proceed to calculate the second variation of Willmore surfaces in $\berger$. Let $F \colon \Sigma \times I \to \berger$ be a normal variation of $\Sigma$ with $\partial_t = f_tN_t$ as in \cref{section: First Variation}, and suppose that $\Sigma=\Sigma_0$ is Willmore. Then, 
 \begin{align}\label{second variation incomplete}
     \left. \frac{d^2}{dt^2} \right|_{t=0} \willmore(\Sigma_t)
     &=\left. \frac{d}{dt} \right|_{t=0} \int_\Sigma f_t \mathcal{E}_\willmore(\Sigma_t) d\Sigma_t \notag \\
     &=\int_\Sigma f \left.\big(\partial_t \mathcal{E}_\willmore(\Sigma_t)\big)\right|_{t=0} d\Sigma. 
 \end{align}
 for $f = f_0$.

\begin{lemma} \label{second variation lemma}
    Let $\varphi \colon \Sigma \times I \to \R$ be a smooth map. The following variation formula holds:
    \begin{equation*}
        \partial_t (\Delta \varphi) = \Delta(\partial_t \varphi) + H g(\nabla \varphi,\nabla f) + f g(\nabla \varphi,\nabla H) -2 \diverg \big( f S(\nabla \varphi) \big).
    \end{equation*}
    In particular,
    \begin{align}
        \partial_t (\Delta H) &= - \Delta \big(J_{\Sigma_t} (f)\big) + H g(\nabla H,\nabla f) + f |\nabla H|^2 \notag \\ 
         & \quad   - 2g\big( \nabla f, S(\nabla H) \big) -2 f \diverg \big( S(\nabla H) \big). \label{variation of laplacian of H}
    \end{align}
    It also holds
    \begin{align}
        \partial_t \big(|\traceless A|^2H\big) &= (H^2 - |\traceless A|^2)J_{\Sigma_t}(f) - 2\big( g(\hess f, A) + f g(\overline{R}_N,A)  + f \tr(A^3) \big)H,  \label{variation of |traceless A|^2 H} \\
        \partial_t C &= -g(T,\nabla f), \label{variation of C} \\
        \partial_t A(T,T) &= f\big( (C^2-1) \overline{K}(T,N) + 2\tau  g\big(S(T),J(T)\big) - |S(T)|^2 \big) \notag \\
        & \quad + 2Cg(S(T),\nabla f) - \hess f (T,T). \label{variation of A(T,T)}
    \end{align} 
    
\end{lemma}

\begin{proof}
From \cref{variation of inverse of the metric,variation of sqrt det g}, it follows that,
    \begin{align*}
        \partial_t (\Delta \varphi) 
        &= \partial_t\big(\frac{1}{\sqrt{\det g}}\partial_i(g^{ij}\sqrt{\det g} \, \partial_j \varphi)\big)  \\
        &=\partial_t(\frac{1}{\sqrt{\det g}}) \partial_i(g^{ij}\sqrt{\det g} \, \partial_j  \varphi) + \frac{1}{\sqrt{\det g}} \partial_i \partial_t (g^{ij}\sqrt{\det g} \, \partial_j \varphi)  \\
        &= -\frac{\partial_t(\sqrt{\det g})}{\det g} \partial_i(g^{ij}\sqrt{\det g} \, \partial_j \varphi) + \frac{1}{\sqrt{\det g}} \partial_i \big( \partial_t (g^{ij} \sqrt{\det g})  \partial_j \varphi \big) + \Delta\partial_t \varphi  \\
        &= -fH \Delta \varphi + \Delta(\partial_t \varphi) + \frac{1}{\sqrt{\det g}} \partial_i \big( -2fg^{ip}g^{jq}A_{pq}\sqrt{\det g} \, \partial_j \varphi + g^{ij} fH \sqrt{\det g} \, \partial_j \varphi   \big)  \\
        &= -fH \Delta \varphi + \Delta(\partial_t \varphi) + \frac{1}{\sqrt{\det g}} \partial_i ( -2fg^{ip}g^{jq}A_{pq}\sqrt{\det g} \, \partial_j \varphi) \notag \\ & \quad + \frac{fH}{\sqrt{\det g}} \partial_i (g^{ij} \sqrt{\det g} \, \partial_j \varphi) + g^{ij} \partial_j \varphi \, \partial_i (fH)  \\
        &= \Delta(\partial_t \varphi) + g(\nabla \varphi,\nabla(fH)) + \frac{1}{\sqrt{\det g}} \partial_i ( -2fg^{ip}g^{jq}A_{pq}\sqrt{\det g} \, \partial_j \varphi)  \\
        &= \Delta(\partial_t \varphi) + H g(\nabla \varphi,\nabla f) + f g(\nabla \varphi,\nabla H) + \frac{-2}{\sqrt{\det g}} \partial_i \Big( f\big(S(\nabla \varphi)\big)^i \sqrt{\det g} \Big)  \\
        &= \Delta(\partial_t \varphi) + H g(\nabla \varphi,\nabla f) + f g(\nabla \varphi,\nabla H) -2 \diverg \big( f S(\nabla \varphi) \big).  
    \end{align*} 
In particular, by \cref{variation of H}
    \begin{align*}
        \partial_t (\Delta H) = - \Delta \big( J_{\Sigma_t}(f) \big) + H g(\nabla H,\nabla f) + f |\nabla H|^2 - 2g\big( \nabla f, S(\nabla H) \big) -2 f \diverg \big( S(\nabla H) \big).
    \end{align*}
The \cref{variation of |traceless A|^2 H} is a consequence of \cref{Squared Norm of Traceless secondff,variation of H,variation of squared 2ff}. Since the variation is normal, from \cref{xi property,grad f} it follows \cref{variation of C}.
    \begin{align}
        \partial_t A(T,T) &= \underbrace{g(\nablaAmbient_{\partial_t} \nablaAmbient_T N,T)}_{\text{(I)} } + \underbrace{g(\nablaAmbient_T N, \nablaAmbient_{\partial_t} T)}_{\text{(II)}} \label{variation of ATT incomplete}.
    \end{align}
From \cref{xi property,variation of C,grad f} we have that,
    \begin{align}
        \nablaAmbient_{\partial_t} T &= \tau f N \wedge \xi + g(T,\nabla f)N  + C \nabla f \notag \\
        &= \tau f J(T) + g(T,\nabla f)N  + C \nabla f. \label{derivative of T in direction of partial_t}
    \end{align}
    Then, 
    \begin{align*}
        (\text{II}) = \tau f g\big( S(T),J(T) \big) + C g\big( S(T),\nabla f \big).
    \end{align*}
    To compute $\text{(I)}$, notice that from \cref{grad f},
    \begin{align*}
        \nablaAmbient_{\partial_t} \nablaAmbient_T N &= \overline{R}(\partial_t,T)N + \nablaAmbient_T \nablaAmbient_{\partial_t} N + \nablaAmbient_{[\partial_t,T]} N \\
        &= -f\overline{R}(T,N)N - \nablaAmbient_T \nabla f + \nablaAmbient_{[\partial_t,T]} N.
    \end{align*}
    Then,
    \begin{align*}
        \text{(I)} = f(C^2-1) \overline{K}(T,N) - \hess f (T,T) + g(\nablaAmbient_{[\partial_t,T]}N,T).
    \end{align*}
    From \cref{derivative of T in direction of partial_t}
    \begin{align*}
        [\partial_t,T]&= \nablaAmbient_{\partial_t} T - \nablaAmbient_T \partial_t \\
        &= \tau f J(T) + g(T,\nabla f)N + C \nabla f -T(f)N - f S(T).
    \end{align*}  
    To further compute \text{(I)},
    \begin{align*}
        g(\nablaAmbient_{[\partial_t,T]}N ,T) &= \tau f g\big(S\big(J(T)\big),T\big) + Cg(S(\nabla f), T) - f g(S^2(T),T)  \\
        &= \tau f g\big(S(T),J(T)\big) + C g(S(T),\nabla f) - f |S(T)|^2.
    \end{align*}
    Hence,
    \begin{align*}
       \text{(I)} = f\big((C^2-1) \overline{K}(T,N) + \tau g\big(S(T),J(T)\big) - |S(T)|^2 \big) + C g(S(T),\nabla f) - \hess f (T,T).
   \end{align*}
   From \cref{variation of ATT incomplete}
   \begin{align*}
       \partial_t A(T,T) &= f\big((C^2 -1) \overline{K}(T,N) + 2\tau g \big(S(T),J(T) \big) - |S(T)|^2 \big) \\
       &\quad + 2C g(S(T),\nabla f) -\hess f (T,T).
   \end{align*}
    
\end{proof}

\begin{theorem}\label{second variation of surfaces in the Berger sphere}
    Let $\Sigma \subset \berger$ be a compact Willmore surface. Then, the second variation of the Willmore energy is
    \begin{align*}
      \left.\frac{d^2}{dt^2}\right|_{t=0} \willmore(\Sigma_t) = -\int_\Sigma f \mathcal{L}_\Sigma (f) d\Sigma, 
    \end{align*}
    where the second derivative operator  $\mathcal{L}_\Sigma$ is given by 
    \begin{align*}
      \mathcal{L}_{\Sigma}(f) 
      &=- \big(J_{\Sigma} \circ (J_\Sigma - 2(2- \tau^2))\big)(f) + 4(\tau^2-1)(1+2C^2)J_{\Sigma}(f) \\ 
      & \quad + 16(\tau^2-1)\Big(f\big((C^2-1)\overline{K}(T,N) + 2\tau g\big(S(T),J(T)\big) - |S(T)|^2 \big)  \\
      & \quad +2C g(S(T),\nabla f) - \hess f (T,T) \Big) \\
      & \quad + f \big(|\nabla H|^2 - 2H g(\overline{R}_N,A)- 2\diverg\big(S(\nabla H) \big) - 2H \tr(A^3)  \big)  \\ 
      & \quad + H\big(g(\nabla H, \nabla f)- 2 g(\hess f, A) + \frac{3}{2}H J_\Sigma(f)  + 8(\tau^2-1)C g(T,\nabla f) \big) \\
      & \quad - 2 g(S(\nabla H), \nabla f)
    \end{align*}
\end{theorem}

\begin{proof}
    From \cref{second variation incomplete,euler-lagrange equation of willmore surfaces in the Berger sphere,second variation lemma} we conclude the calculation of the second variation.
\end{proof}
 \section{Hopf tori in Berger spheres}\label{section: hopf tori}
We study a classical family of surfaces in Berger spheres, namely the Hopf tori. In particular, to the best of our knowledge, we recover all known explicit examples of Willmore tori in $\berger$, as described in \cite{Manuel1997,Manuel2014}.

\medskip
Let $\beta : I \hookrightarrow \Sphere^2\left(\frac{1}{2}\right)$ be a smooth regular closed curve, defined on a compact interval $I$, and denote $\hopftori = \pi^{-1}(\beta(I))$, where $\pi$ is the Hopf fibration. The family of immersed surfaces $\{\hopftori\}_\beta$ is known as the \emph{Hopf tori} family. We endow $\hopftori$ with the metric induced from $\berger$. Each $\hopftori \subset \berger$ is a flat torus, which is embedded whenever $\beta$ is simple. Note that $\xi$ is tangent to $\hopftori$. 

\medskip
When $\beta$ parametrizes the equator $\Sphere^2\left(\frac{1}{2}\right) \cap (\C \times \{0\})$, the surface $\hopftori$ is the Clifford torus, which we denote by $\Sigma_\tau$. The following Hopf torus will also be relevant. Let $\tilde{\beta}$ parametrize the horizontal circle $\Sphere^2\left(\frac{1}{2}\right) \cap (\C \times \{\frac{1}{2}\sqrt{\frac{\tau^2-2}{\tau^2-1}}\})$ of constant geodesic curvature $2\sqrt{\tau^2-2}$ in $\Sphere^2\left(\frac{1}{2}\right)$. We denote $\widetilde{\Sigma}_\tau = \Sigma^{\tilde{\beta}}$. We observe that $\widetilde{\Sigma}_\tau$ = $\Sphere^1(a) \times \Sphere^1(b)$, where 
\begin{align}\label{eq: constantes do irmao bifurcado}
    a = a(\tau)  = \sqrt{\frac{1}{2}+c(\tau)}, \quad b = b(\tau) = \sqrt{\frac{1}{2}-c(\tau)}, \quad \text{and} \quad c = c(\tau) = \frac{1}{2}\sqrt{\frac{\tau^2-2}{\tau^2-1}},
\end{align} 
are the radii of the circles, and the height of the horizontal circle parametrized by $\tilde{\beta}$, respectively. As we will see below, $\widetilde{\Sigma}_\tau$ is a Willmore surface bifurcating from the Clifford torus. For this reason, we refer to it as the \emph{bifurcated brother}.

\medskip
We proceed to compute the shape operator of $\hopftori$ in $\berger$. Suppose that $\beta$ is parametrized by the arc length. As the computation is local, we may assume that $\beta$ is embedded in order to simplify the notation. Let the spheres $\Sphere^2\left(\frac{1}{2}\right)$ and $\Sphere^3$ be oriented by the inner unit normal vector field. Let $W$ be a local vector field of $\Sphere^2\left(\frac{1}{2}\right)$ extending $\beta'$ and $\vec{n}$ be the inner unit normal orienting $\Sphere^2\left(\frac{1}{2}\right)$. Let $N=\vec{n} \times W$, where $\times$ is the vector product of $\R^3$, be the unit vector field tangent to $\Sphere^2\left(\frac{1}{2}\right)$, orthogonal to $\beta$, and such that $k_g^\beta=<\beta'',N \circ \beta>$ is the geodesic curvature of $\beta$. Let $\overline{W}$ and $\overline{N}$ be the horizontal lifts of $W$ and $N$ with respect to $\pi$. Then, $(\overline{W},\xi,\overline{N})$ and $(\overline{W}, \xi) $ are positively oriented orthonormal local frames of $T\berger$ and $T\hopftori$. In particular, this is the orientation defining the vector product where \cref{xi property} holds.

\begin{lemma}\label{hopf tori shape-operator}
   Let $S$ be the shape operator of $\hopftori$ with respect to $\overline{N}$. Then, the matrix representation of $S$ in the orthonormal frame $(\overline{W},\xi)$ is given by
   \begin{equation*}
       \begin{pmatrix}
       -k_g^\beta & -\tau \\
       -\tau & 0 
   \end{pmatrix}.
   \end{equation*}
\end{lemma}
\begin{proof}
From \cref{xi property}, and the fact that the Hopf fibration is a Riemannian submersion, we obtain:
    \begin{align*}
        g(S(\xi),\xi) &=-g(\overline{N},\nablaAmbient_\xi \xi) = 0, \\
        g(S(\overline{W}),\xi) &= -g(\overline{N}, \nablaAmbient_{\overline{W}} \xi)=-g(\overline{N}, \tau \overline{N}) = -\tau, \\
        g(S(\overline{W}),\overline{W}) &= g(\nablaAmbient_{\overline{W}} \overline{N}, \overline{W}) = \ip{\nabla_W N}{W}=-k_g^\beta.
    \end{align*}
    
\end{proof}

Recall that an elastica curve is a critical point of the functional 
\begin{equation*}
    \mathcal{F}_\lambda(\beta) = \int_I (k_g^\beta(t))^2 + \lambda \, dt,
\end{equation*}
where $\lambda \in \R$ is a constant playing the role of a Lagrange multiplier.
Langer and Singer classified elastica curves in the round $2$-sphere \cite{LangerSinger1984}.

\begin{corollary}\label{corollary: Classification of Hopf Willmore surfaces}
    The Hopf torus $\hopftori$ is a Willmore surface of $\berger$ precisely when the curve $\beta$ is an elastica of $\mathcal{F}_\lambda$ with Lagrange multiplier $\lambda = 4\tau^2$.
\end{corollary}

\begin{proof}
    By \cref{euler-lagrange equation of willmore surfaces in the Berger sphere,hopf tori shape-operator}, since the fibres of $\pi$ are totally-geodesic in $\berger$, using O’Neill’s covariant derivative formulas we obtain that $\hopftori$ is Willmore exactly when
    \begin{align}
        (k_g^\beta)'' + (k_g^\beta)^3 + (8-4\tau^2)k_g^\beta = 0, \label{euler-lagrange equation of elastica}
    \end{align}
    which is the Euler-Lagrange equation of $\mathcal{F}_{4\tau^2}$ by \cite{LangerSinger1984}.
    
\end{proof}

\begin{remark}\label{remark: trivial bifurcation at sqrt 2}
    The Euler--Lagrange equation \cref{euler-lagrange equation of elastica} admits two trivial solutions, namely 
    \begin{equation*}
        k_g^\beta = 0 \quad \text{and} \quad k_g^{\tilde{\beta}} = 2\sqrt{\tau^2-2} \text{  (when  } \tau>\sqrt{2}).
    \end{equation*}
    The former corresponds to the Clifford torus $\Sigma_\tau$. The latter corresponds to $\widetilde{\Sigma}_\tau = \Sigma^{\tilde{\beta}}$ when $\tilde{\beta}$ parametrizes the circle $\Sphere^2\left(\frac{1}{2}\right) \cap (\C \times \{\frac{1}{2}\sqrt{\frac{\tau^2-2}{\tau^2-1}} \})$. Notice that $\widetilde{\Sigma}_\tau$ approaches $\Sigma_\tau$ as $\tau \searrow \sqrt{2}$ and goes to the north pole fibre $\pi^{-1}\big((0,\frac{1}{2})\big) \subset \berger$ as $\tau \nearrow +\infty$. This is one of the two bifurcations occurring in \cref{theorem: Bifurcations of the Clifford Torus} that we can describe explicitly (see \cref{remark: trivial bifurcation at 1} for the description of the other one, which is trivial). We also note that $\widetilde{\Sigma}_\tau$ is an example of a non-minimal CMC Willmore surface in the range $\tau \in (\sqrt{2}, \infty)$, contrasting with \cref{Prop: Non-existence of CMC Willmore surfaces}.
\end{remark}

\begin{remark}\label{remark: Willmore energy of Hopf tori family}
From the fact that the Hopf map $\berger \to \Sphere^2\left(\frac{1}{2}\right)$ is a Riemannian submersion whose fibres have length equal to $2\pi\tau$, applying \cref{hopf tori shape-operator} we compute the Willmore energies of a Hopf torus $\hopftori$:
\begin{align}
    \willmore(\hopftori) &= 2\pi \tau \int_I \frac{(k_g^{\beta})^2}{2} + 2\tau^2 \, ds \notag  \\
    &=\pi\tau \mathcal{F}_{4\tau^2}(\beta). \label{formula: Willmore energy of hopf torus}
\end{align}
In particular, $\Sigma_\tau$ and $\widetilde{\Sigma}_\tau$ have Willmore energies
\begin{align*}
\willmore(\Sigma_\tau) = 4 \pi^2 \tau^3 \ \ \text{and} \ \ \willmore(\widetilde{\Sigma}_\tau)=8\pi^2\tau\sqrt{\tau^2-1}.
\end{align*}
As soon as $\widetilde{\Sigma}_\tau$ exists, that is, $\tau>\sqrt{2}$, we have $\willmore(\widetilde{\Sigma}_\tau)<\willmore(\Sigma_\tau)$. In view of \cref{theorem: Stability of the Clifford torus}, $\Sigma_\tau$ is not a local minimum of the Willmore energy when $\tau>1$, since it is unstable. By \cref{theorem: instabilidade do irmao paralelo}, $\widetilde{\Sigma}_\tau$ never locally minimizes the Willmore energy for the same reason. In \cite[Th. 4.3.]{Langer&Singer1987}, it is shown that all the non--geodesic elasticae of $\mathcal{F}_\lambda$ are unstable when $0<\lambda \le \frac{32}{7}$. Hence, by \cref{corollary: Classification of Hopf Willmore surfaces,formula: Willmore energy of hopf torus}, if $\hopftori \ne \Sigma_\tau$ is a Willmore surface and $0<\tau\leq \sqrt{\frac{8}{7}}$, then a variation that decreases the elastic energy of $\beta$ pulls back via $\pi$ to a variation of $\hopftori$ that decreases its Willmore energy. Moreover, by the Langer and Singer's description of elasticae and \cref{formula: Willmore energy of hopf torus}, we have $\willmore(\Sigma_\tau)<\willmore(\hopftori)$ when the geodesic curvature $k_g^\beta$ is non--constant and $0<\tau\leq \sqrt{\frac{8}{7}}$. For the comparison of the Willmore energies of the Clifford torus and the equator, see \cref{section: Final Remarks}. 
\end{remark}

\begin{proposition}\label{Clifford torus second variation operator}
    For each $\tau>0$, the second variation operator $\mathcal{L}_{\Sigma_\tau}$ of the Clifford torus $\Sigma_\tau$ is given by
    \begin{align*}
        \mathcal{L}_{\Sigma_\tau} = -\Delta^2 - 2(4-\tau^2)\Delta + 16(1-\tau^2) L_\xi^2 -8(2-\tau^2),
    \end{align*}
where $L_\xi$ is the Lie derivative in the direction of $\xi$ and $L_\xi^2 = L_\xi \circ L_\xi$.
\end{proposition}

\begin{proof}
Since $\xi$ is tangent to the minimal surface $\Sigma_\tau = \hopftori \hookrightarrow \berger$ with $k_g^\beta=0$. From \cref{Riemann tensor of the Berger sphere,Ricci tensor of the berger Sphere,xi property,hopf tori shape-operator} we obtain:
\begin{align*}
        C &= 0, \\
        \overline{K}(\xi,\overline{N}) &= \tau^2, \\
        \overline{\ric}(\overline{N},\overline{N}) &= 2(2-\tau^2), \\
        |A|^2 &=  2 \tau^2, \\
        g(S(\xi), J(\xi)) &= \tau, \\
        |S(\xi)|^2 &= \tau^2, \\
        \hess f (\xi,\xi) &= L_\xi^2 f.
    \end{align*}
The result follows from \cref{second variation of surfaces in the Berger sphere}.

\end{proof}

\begin{proposition}\label{proposition: spectrum of the Clifford torus}
The eigenvalues of the second variation operator $\mathcal{L}_{\Sigma_\tau}$ have the form 
\begin{align*}
\lambda_{m,n}^{\pm}(\tau) & = 
(m \,\pm\, n)^4\frac{1}{\tau^4}  + 2(m\,\pm\,n)^2\big((m\,\mp\,n)^2 + 4\big)\frac{1}{\tau^2}  \\
& + \Big((m\,\mp\,n)^4 - 8(m\,\mp\,n)^2 - 14(m\,\pm\,n)^2 + 16\Big) \\
& + \Big(2(m\,\mp\,n)^2 - 8\Big)\tau^2 \,,
\end{align*}
where $m,n \in \N_0 = \N \cup \{0\}$. The functions
\begin{align*}
      \tfrac{1}{\sqrt{2}}(e^{i\theta},e^{i\varphi}) \mapsto \cos(m\theta \pm n\varphi), \, \sin(m\theta \pm n\varphi) \,,
\end{align*}
are eigenfunctions generating the eigenspace corresponding to $\lambda^{\pm}_{m,n}(\tau)$ and span the Sobolev space $H^4(\Sigma_\tau)$.
\end{proposition}

\begin{remark}
    There exist $(m,n) \neq (k,l)$ in $\N_0^2$ and $\tau>0$ such that $\lambda^+_{m,n}(\tau) = \lambda^+_{k,l}(\tau) \,.$
\end{remark}
\begin{proof}
It is well known that the functions in the statement span $H^4(\Sigma_\tau)$. To prove the remainder of the statement, denote by $\Delta_\tau$ the Laplacian operator of $\Sigma_\tau$. Notice that $\Delta_1$ is the Laplacian operator of $\Sigma_1$ when equipped with the canonical product metric. We proceed to compute the eigenvalues of $\mathcal{L}_{\Sigma_\tau}$. Let $\R^2\to \Sigma_\tau$ be the universal cover of the Clifford torus given by 
   \begin{equation*}
       (\theta,\varphi) \mapsto \tfrac{1}{\sqrt{2}}(e^{i\theta},e^{i\varphi}).
   \end{equation*}
   In these coordinates,
   \begin{align*}
    [g]_{(\partial_\theta,\partial_\varphi)} &= \frac{1}{4}
       \begin{pmatrix}
           \tau^2+1 & \tau^2-1 \\
           \tau^2-1 & \tau^2+1
       \end{pmatrix}, \\
       [g]^{-1}_{(\partial_\theta,\partial_\varphi)} &= \frac{1}{\tau^2}
       \begin{pmatrix}
           \tau^2+1 & 1-\tau^2 \\
           1-\tau^2 & \tau^2+1
       \end{pmatrix}, \\
       L_\xi &= \frac{1}{\tau}(\partial_\theta + \partial_\varphi), \\
       \Delta_1 &= 2(\partial_\theta^2+ \partial_\varphi^2), \\
       \Delta_\tau &= \frac{\tau^2+1}{\tau^2}(\partial_\theta^2 + \partial_\varphi^2) + 2 \frac{1-\tau^2}{\tau^2}\partial_\theta \partial_\varphi = \Delta_1 +(1-\tau^2)L_\xi^2.
   \end{align*}
   Since $\Delta_1$ and $L_\xi$ have constant coefficients in the base of coordinate vector fields, they commute and $L_\xi^2$ preserves each eigenspace of $\Delta_1$. Computing in local coordinates, it readily follows that $\diverg_{\Sigma_\tau} \xi =0$. Thus, $L_\xi^2$ is symmetric and we can decompose the eigenspaces of $\Delta_1$ in those of $L_\xi^2$. As it is well known, by the separation of variables method, we find the eigenvalues of $\Delta_1$,
   \begin{equation*}
       \Delta_1f +\mu_{m,n} f = 0, \quad  \text{for $f \in E(\Delta_1,\mu_{m,n})$},
   \end{equation*}
   where $\mu_{m,n} = 2(m^2+n^2)$ and $E(\Delta_1,\mu_{m,n})$ is the eigenspace of $\Delta_1$ corresponding to the eigenvalue $\mu_{m,n}$. This eigenspace is generated by the functions $\Sigma_\tau \to \R$, which in coordinates are given by
   \begin{align*}
       \cos(k\theta)\cos(l\varphi), \, \cos(k\theta)\sin(l\varphi), \,
       \sin(k\theta)\cos(l\varphi), \,
       \sin(k\theta)\sin(l\varphi), 
   \end{align*}
   for all $k,l \in \N_0$ such that $k^2+l^2 = m^2 + n^2$. We diagonalize $L_\xi^2$ restricted to $E(\Delta_1,\mu_{m,n})$ and find that
   \begin{equation*}
       L_\xi^2 f + \nu_{m,n}^{\pm} f  = 0, \quad \text{$f \in E(L_\xi^2,\nu_{m,n}^\pm)$},
   \end{equation*}
   where $\nu_{m,n}^\pm = \frac{1}{\tau^2}(m \, \pm \,n)^2.$ The eigenspace $E(L_\xi^2,\nu_{m,n}^+)$ is generated by the functions $\Sigma_\tau \to \R$, which in coordinates are given by
   \begin{align*}
       \cos(k\theta+l\varphi), \, 
        \sin(k\theta +l\varphi),
   \end{align*}
   where $k,l \in \N_0$ satisfy $k+l = m + n$. The eigenspace $E(L_\xi^2,\nu_{m,n}^-)$ is generated by
   \begin{align*}
       \cos(k\theta-l\varphi), \,
        \sin(k\theta -l\varphi),
   \end{align*}
   where $k,l \in \N_0$ satisfy $k-l = m-n$. By \cref{Clifford torus second variation operator}, we are done.

\end{proof}
We now have all the necessary tools to prove \cref{theorem: Stability of the Clifford torus}.
\begin{proof}[Proof of \cref{theorem: Stability of the Clifford torus}] The proof follows from \cref{proposition: spectrum of the Clifford torus} and elementary computations.

\medskip
The eigenvalue $\lambda_{1,1}^{-}(\tau)$ is zero for all $\tau>0$, with multiplicity $2$ when $0<\tau<1$. For $0<\tau<1$, all the eigenvalues $\lambda_{m,n}^{\pm}$, besides $\lambda_{1,1}^{-}$, are positive. When $\tau=1$, only $\lambda_{1,0}^{\pm}=\lambda_{0,1}^{\pm}=\lambda_{1,1}^\pm=0$. Hence, $n_1(\Sigma_1)=8$. Also, $\lambda_{1,0}^+(\tau)<0$ for $\tau>1$. Thus, $\Sigma_\tau$ is stable exactly when $\tau \in (0,1]$.

\medskip
Let
\begin{equation*}
    \mathcal{D} = \{ \tau \in (1,\infty) \mid \lambda_{m,n}^\pm(\tau) = 0 \ \text{for some} \ \lambda_{m,n}^\pm \notin \{ \lambda_{0,0}^+, \ \lambda_{0,0}^-, \ \lambda_{1,1}^- \} \}.
\end{equation*}
The roots of $\lambda^{\pm}_{m,n}$ coincide with those of the polynomial $\tau^4 \lambda^{\pm}_{m,n}$. The claims about $\mathcal{D}$ and the nullity $n_{\sqrt{2}}(\Sigma_{\sqrt{2}})=3$ follow from the next observation. Note that $\lambda_{m,n}^- \ge \lambda_{m,n}^+$ and using polar coordinates $m = r \cos \theta$ and $n = r \sin \theta$, one may verify that $\lambda_{m,n}^+(\sqrt{2})>0$ whenever $m^2+n^2>5$. This observation reduces the analysis to a finite number of eigenvalues. Among all $\lambda_{m,n}^\pm$, only 
$$
\lambda_{0,0}^+=\lambda_{0,0}^-=\lambda_{1,1}^-=0,
$$
at $\tau = \sqrt{2}$.

\medskip
It remains to show  $i_\tau(\Sigma_\tau) \ge \left\lfloor \tau^2 \right\rfloor$ whenever $\tau > 1$. Indeed, for each $m \in \N$, we see that
\begin{align*}
    \lambda_{m,m}^+(\tau)
    &= \frac{16m^4}{\tau^4} + \frac{32m^2}{\tau^2} +(16-56m^2) -8\tau^2 , \\
    \lim_{\tau \to 0^{+}} \lambda_{m,m}^+(\tau) 
    &=  \infty, \\
    \lim_{\tau \to \infty} \lambda_{m,m}^+(\tau) 
    &=-\infty, \\
    \frac{d\lambda_{m,m}^+}{d\tau}(\tau) 
    &=-\frac{64m^4}{\tau^5}-\frac{64m^2}{\tau^3}-16\tau <0, \\
    \lambda_{m,m}^+(\sqrt{m}) 
    &=-40m^2+24m+16<0, \, \text{when $m>1$.}
\end{align*}
The equation $\lambda_{m,m}^+(\tau_m)=0$ has a unique solution $\tau_m>0$. Since $\lambda_{m,m}^+(\sqrt{m})<0$, we obtain $\tau_m<\sqrt{m}$. Therefore, $i_\tau(\Sigma_\tau) \ge \left\lfloor \tau^2 \right\rfloor$.

\end{proof}
\begin{remark}\label{remark: behavior of the sequence tau_m}
    From the proof of \cref{theorem: Stability of the Clifford torus}, we conclude that, for each $m \in \N_0$, there exists a unique $\tau_m>0$ such that $\lambda_{m,m}^+(\tau_m)=0$, where $\lambda_{m,m}^+$ is given in \cref{proposition: spectrum of the Clifford torus}. It turns out that the sequence $\{\tau_m\}_{m \in \N_0}$ is precisely the set of all $K$-invariant bifurcation parameters of the Clifford torus as a Willmore surface, in the sense of \cref{section: Bifurcations of the Clifford torus}. We remark that 
   \begin{gather}\label{Clifford bifurcations parameters}
      \begin{cases}
         \tau_0 = \sqrt{2}, \ \tau_1 = 1 ,\\ 
         \tau_3 < \sqrt{2} < \tau_4, \\
         \tau_n < \tau_m \quad \text{for} \quad 1 \leq n < m, \\ 
         \tau_m \nearrow +\infty.
      \end{cases}
   \end{gather}
\end{remark}
Even though we do not have a closed formula for $i_\tau(\Sigma_\tau)$, \cref{proposition: spectrum of the Clifford torus} allows us to compute, for example, the index of $\Sigma_\tau$ when $\tau = \sqrt{2}$. Indeed, by the same observation as in the above proof, since $\lambda_{m,n}^+(\sqrt{2}) < 0$ only for $0 \leq m,n \leq 3$ and $(m,n) \neq (0,0),(0,3),(3,0)$, we conclude that $i_\tau(\Sigma_\tau) = 34$.
\begin{proposition}\label{proposition: second variation operator irmão paralelo}
    For each $\tau>\sqrt{2}$, the second variation operator of the Willmore surface $\widetilde{\Sigma}_\tau$ is given by
    \begin{equation*}
       \mathcal{L}_{\widetilde{\Sigma}_\tau} = -\Delta^2 - 4\Delta  + 16(1-\tau^2)L_{\xi}^2 -2H g(\hess({\cdot}),A) - 16(2-\tau^2)(1-\tau^2) ,
    \end{equation*}
\end{proposition}
\begin{proof}
Since $\xi$ is tangent to the CMC surface $\Sigma_\tau \hookrightarrow \berger$, from \cref{Riemann tensor of the Berger sphere,Ricci tensor of the berger Sphere,xi property,hopf tori shape-operator}
we obtain:
\begin{align*}
        C &= 0, \\
        \overline{K}(\xi,N) &= \tau^2, \\
        \overline{\ric}(N,N) &= 2(2-\tau^2), \\
        |A|^2 &=  6\tau^2 - 8, \\
        g(S(\xi), J(\xi)) &= \tau, \\
        |S(\xi)|^2 &= \tau^2, \\
        \hess f (\xi,\xi) &= L_\xi^2 f, \\
        g(\overline{R}_{N},A) &= (4-3\tau^2)H.
\end{align*}
The result follows from the above equations, the algebraic identity
\begin{equation*}
    2\tr(A^3)=H(|A|^2+2|\traceless A|^2),
\end{equation*}
and \cref{second variation of surfaces in the Berger sphere}.

\end{proof}

\begin{proposition}\label{proposition: spectrum of the irmao paralelo}
The eigenvalues of the second variation operator $\mathcal{L}_{\widetilde{\Sigma}_\tau}$ have the form 
\begin{align*}
\tilde{\lambda}_{m,n}^{\pm}(\tau) =&
\big(\frac{m^2}{a^2}+\frac{n^2}{b^2} + \frac{1-\tau^2}{\tau^2}(m \pm n)^2 \big)^2 - 4\big(\frac{m^2}{a^2}+\frac{n^2}{b^2} + \frac{1-\tau^2}{\tau^2}(m \pm n)^2\big) \\
&+ 16\frac{1-\tau^2}{\tau^2}(m \pm n)^2 + 2H (\frac{b}{a^3}m^2 - \frac{a}{b^3}n^2) - 16(\tau^2-2)(1-\tau^2),
\end{align*}
where $m,n \in \N_0$, $H = 2\sqrt{\tau^2-2}$, $a = \sqrt{\frac{1}{2} + c}$, $b = \sqrt{\frac{1}{2} - c}$ and $c = \frac{1}{2} \sqrt{\frac{\tau^2-2}{\tau^2-1}}$. The functions
\begin{align*}
      (ae^{i\theta},be^{i\varphi}) \mapsto \cos(m\theta \pm n\varphi), \, \sin(m\theta \pm n\varphi),
\end{align*}
are eigenfunctions generating the eigenspace corresponding to $\tilde{\lambda}^{\pm}_{m,n}(\tau)$ and span the Sobolev space $H^4(\widetilde{\Sigma}_\tau)$.
\end{proposition}

\begin{proof}
     We have that $\widetilde{\Sigma}_\tau = \Sphere^1(a) \times \Sphere^1(b)$. To verify the computations below it will be useful to have in mind the relations given by \cref{eq: constantes do irmao bifurcado},
    \begin{align*}
      \begin{cases}
        a^2+b^2 = 1, \\
        a^2-b^2 = 2c.
      \end{cases}
    \end{align*}
    Denote by $\Delta_\tau$ the Laplacian operator of $\widetilde{\Sigma}_\tau$. Notice that $\Delta_1$ is the Laplacian operator of $\widetilde{\Sigma}_1$ when equipped with the canonical product metric. Consider the universal cover
    \begin{equation*}
         \R^2 \to \widetilde{\Sigma}_\tau, \quad (\theta,\varphi) \mapsto (ae^{i\theta},be^{i\varphi}).
    \end{equation*}
    In these coordinates, 
    \begin{align*}
        \xi(\theta,\varphi) &= \frac{1}{\tau}(\partial_\theta + \partial_\varphi), \\
        [g]_{(\partial_\theta,\partial_\varphi)} &=
        \begin{pmatrix}
           a^2(b^2 + a^2\tau^2) & a^2b^2(\tau^2-1) \\
           a^2b^2(\tau^2-1) & b^2(a^2+b^2\tau^2)
        \end{pmatrix}, \\
        [g]_{(\partial_\theta,\partial_\varphi)}^{-1} &= \frac{1}{r^2\tau^2} 
    \begin{pmatrix}
       r^2 + b^4\tau^2 &  r^2(1-\tau^2) \\
       r^2(1-\tau^2)   &  r^2+a^4\tau^2)
    \end{pmatrix}, \\
    \Delta_1 &= \frac{1}{a^2}\partial_\theta^2 + \frac{1}{b^2}\partial_\varphi^2, \\
    \Delta_\tau &= \Delta_1 + (1-\tau^2)L_\xi^2,
    \end{align*}
where $r = \sqrt{\frac{1}{4} - c^2}$. Since $\nabla_{\partial_\theta}\partial_\varphi = 0$ and $ \nabla_{\partial_\varphi}\partial_\theta = 0$, we obtain
\begin{equation*}
    [\hess (\cdot)]_{(\partial_\theta,\partial_\varphi)} = 
    \begin{pmatrix}
        \partial_\theta^2  & \partial_\theta\partial_\varphi  \\
        \partial_\theta\partial_\varphi  & \partial_\varphi^2 
    \end{pmatrix}.
\end{equation*}
In order to compute the second fundamental form $A$ in coordinates, let $\tilde{\beta}$ parametrize the circle $\Sphere^2\left(\frac{1}{2}\right) \cap (\C \times \{\frac{1}{2}\sqrt{\frac{\tau^2-2}{\tau^2-1}}\})$ with orientation such that $k_g^{\tilde{\beta}}= - 2\sqrt{\tau^2-2}$. Let $\overline{W} \in \campos(\widetilde{\Sigma}_\tau)$ be as in the beginning of \cref{section: hopf tori}. We have
\begin{align*}
    \overline{W} &= \frac{b^2}{r} \partial_\theta - \frac{a^2}{r}\partial_\varphi, \\
    \partial_\theta &= r\overline{W} + a^2\tau \xi, \\
    \partial_\varphi &= -r \overline{W} + b^2\tau \xi.
\end{align*}
By \cref{hopf tori shape-operator}, we get
\begin{equation*}
    [A]_{(\partial_\theta,\partial_\varphi)} = 2r
    \begin{pmatrix}
        c-a^2\tau^2 & c(\tau^2-1) \\
        c(\tau^2-1) & c+b^2\tau^2
    \end{pmatrix}.
\end{equation*}
Hence,
\begin{equation*}
    g_\tau(\hess (\cdot), A) = -\frac{b}{a^3} \partial_\theta^2 + \frac{a}{b^3} \partial_\varphi^2.
\end{equation*}
By the same argument as in the proof of \cref{proposition: spectrum of the Clifford torus}, we can decompose the eigenspaces of $\Delta_1$ in those of $L_\xi^2$. By the separation of variables method, we find the eigenvalues of $\Delta_1$.
\begin{equation*}
    \Delta_1f +\mu_{m,n} f = 0, \quad  \text{for $f \in E(\Delta_1,\mu_{m,n})$,}
\end{equation*}
where $\mu_{m,n} = \frac{m^2}{a^2}+\frac{n^2}{b^2}$ and $E(\Delta_1,\mu_{m,n})$ is the eigenspace of $\Delta_1$ corresponding to the eigenvalue $\mu_{m,n}$. This eigenspace is generated by the functions $\widetilde{\Sigma}_\tau \to \R$ which in coordinates are given by
\begin{align*}
    \cos(k\theta)\cos(l\varphi), \, \cos(k\theta)\sin(l\varphi), \, \sin(k\theta)\cos(l\varphi), \, \sin(k\theta)\sin(l\varphi), 
\end{align*}
for all $k,l \in \N_0$ such that $\frac{k^2}{a^2}+\frac{l^2}{b^2} = \mu_{m,n}$. We diagonalize $L_\xi^2$ restricted to $E(\Delta_1,\mu_{m,n})$ and find that
\begin{equation*}
    L_\xi^2 f + \nu_{m,n}^{\pm} f  = 0, \quad \text{$f \in E(L_\xi^2,\nu_{m,n}^\pm)$},
\end{equation*}
where $\nu_{m,n}^\pm = \frac{1}{\tau^2} (m \, \pm \,n)^2.$ The eigenspace $E(L_\xi^2,\nu_{m,n}^+)$ is generated by the functions $\widetilde{\Sigma}_\tau \to \R$ which in coordinates are given by
\begin{align*}
   \cos(k\theta+l\varphi), \, \sin(k\theta +l\varphi),
\end{align*}
where $k,l \in \N_0$ satisfy $\frac{1}{\tau^2}(k^2+l^2)= \nu_{m,n}^+$. The eigenspace $E(L_\xi^2,\nu_{m,n}^-)$ is generated by
\begin{align*}
   \cos(k\theta-l\varphi), \,
   \sin(k\theta -l\varphi),
\end{align*}
where $k,l \in \N_0$ satisfy $\frac{1}{\tau^2}(k^2-l^2)= \nu_{m,n}^-$. It readily follows from the coordinate expression of $g_\tau(\hess (\cdot), A)$ that 
\begin{equation*}
    g_\tau(\hess (f), A) + \rho_{m,n}f = 0, \quad f \in E(g_\tau(\hess (\cdot), A), \rho_{m,n}),
\end{equation*}
where $\rho_{m,n} = -\frac{b}{a^3}m^2 +\frac{a}{b^3}n^2$. The eigenspace $E(g_\tau(\hess (\cdot), A), \rho_{m,n})$ is generated by the functions $\widetilde{\Sigma}_\tau \to \R$ which in coordinates are given by
\begin{align*}
   \cos(k\theta+l\varphi), \, \sin(k\theta +l\varphi), \cos(k\theta-l\varphi), \, \sin(k\theta -l\varphi), 
\end{align*}
where $k,l \in \N_0$ satisfy $-\frac{b}{a^3}k^2 +\frac{a}{b^3}l^2 = \rho_{m,n}$. By the above equations and \cref{proposition: second variation operator irmão paralelo}, we are done.

\end{proof}

\begin{proof}[Proof of \cref{theorem: instabilidade do irmao paralelo}]
By \cref{proposition: spectrum of the irmao paralelo} and a elementary simplification, we verify that 
\begin{equation*}
    \tilde{\lambda}_{1,1}^+(\tau) = -\frac{16}{\tau^4}(\tau^2-1)(1+2\tau^4)<0.
\end{equation*}
Hence, $\widetilde{\Sigma}_\tau$ is unstable.
By \cref{eq: constantes do irmao bifurcado}, we obtain
\begin{align*}
    c(\tau) &= \frac{1}{2} -\frac{1}{4\tau^2} + O(\tau^{-4}), \\
    a(\tau) &= 1 - \frac{1}{8\tau^2} + O(\tau^{-4}), \\
    b(\tau) &= \frac{1}{2\tau} + O(\tau^{-3}), \\
    H(\tau) &= 2\tau - \frac{2}{\tau} + O(\tau^{-3}),
\end{align*}
as $\tau \nearrow +\infty$. Substituting in the expression of $\tilde{\lambda}_{m,1}^+$ and using 
\begin{align*}
 \begin{cases}
    a^2+b^2=1, \\
    a^2-b^2=2c, \\
    a^2b^2= \frac{1}{4(\tau^2-1)},
 \end{cases}
\end{align*}
we see that for each $m \in \N$
\begin{equation*}
    \lambda_{m,1}^+(\tau) = -16(m+1)\tau^2 + O(1).
\end{equation*}
Thus, $i_\tau(\widetilde{\Sigma}_\tau) \nearrow +\infty$ as $\tau \nearrow +\infty.$

\end{proof}
 \section{Bifurcations of the Clifford torus and the bifurcated brother} \label{section: Bifurcations of the Clifford torus}

Before stating the main theorem, we set up the abstract framework in which the concrete cases of the Clifford torus and the bifurcated brother will be treated. Denote $\T^2 = \Sphere^1 \times \Sphere^1 \subset \C^2$. Let $I \subset \R$ be an interval, $v^{\tau}: \T^2 \to \berger$, $\tau \in I$, be a smooth family of embedded Willmore tori, $X_\tau = \ima(v^\tau)$ endowed with the metric induced by $\berger$, and $n_\tau$ the unit normal vector field along $v^\tau$. Denote the group of symmetries of $X_\tau$ by
$$
\fix(X_\tau)=\{ \psi \in \iso (\berger): \psi(X_\tau)=X_\tau \}.
$$
Suppose $G$ is a subgroup of $\fix(X_\tau)$, for every $\tau \in I$, that preserves the orientation of $\berger$. Each symmetry $k \in G$ and each parameter $\tau \in I$ defines a diffeomorphism $\varphi^{v,\tau}_k \colon \T^2 \to \T^2$ by the equation
\begin{align*}
   k \circ v^\tau = v^\tau \circ \ \varphi^{v,\tau}_k.
\end{align*}
Assume that $\varphi^{v,\tau}_k = \varphi^v_k$ does not depend on $\tau$. A function $f \colon \T^2 \to \R$ is said to be $G$--\emph{invariant} with respect to $v$ if
$$
\forall k \in G, \quad f \circ \varphi^v_k = f.
$$
Given a function $f$ on $\T^2$, we define the parametrization of the graph of $f$ over $X_\tau$ by 
\begin{align*}
   v_f^\tau \colon \T^2 \to \berger, \quad p \mapsto \exp_{v^\tau(p)}^\tau\big(f(p)n_\tau(p)\big),
\end{align*}
where $\exp^\tau$ denotes the exponential map of $\berger$. As isometries commute with the exponential map, for each $k\in G$, one may verify that 
$$
k \circ v^\tau_f = v^\tau_f \circ \varphi_k^v,
$$
whenever $f$ is $G$--invariant with respect to $v$, since $k$ preserves orientation by assumption. In particular, $G$ is also a subgroup of the symmetry group $\fix(\ima(v^\tau_f))$. In that case, we say that $v^\tau_f$ is G--\emph{invariant}. 

\medskip
Fix $\alpha \in (0,1)$. For each $l \in \N_0 = \N \cup\{0\}$, we denote by $C^{l,\alpha}_{v,G}(\T^2)$ the subspace of $G$--invariant functions in the Hölder space $C^{l,\alpha}(\T^2)$ with respect to $v$. This subspace is closed, thus it is also a Banach space.
\begin{definition}\label{definition: invariant bifurcation}
   We say that $\tau$ is a $G$--\emph{invariant bifurcation parameter} for $X_\tau$ as a Willmore surface if there exist $\delta>0$, an open neighbourhood $\mathcal{U} \subset (0,+\infty) \times C^{4,\alpha}_{v,G}(\T^2)$ of $(\tau,0)$, and a continuous path 
    \begin{equation*}
        \gamma \colon (-\delta, \delta) \to \mathcal{U}, \quad \gamma(t) = \big(\tau(t),f(t)\big),
    \end{equation*} 
    such that
    \begin{enumerate}
        \item[(I)] $\tau(0) = \tau$, $f(0)=0 \in C^{4,\alpha}_{v,G}(\T^2)$, and $f(t) \ne 0$ whenever $t \ne 0$.
        \item[(II)] For every $t \in (-\delta,\delta)$, the graphical embedding $v^{\tau(t)}_{f(t)} \colon \T^2 \to \Sphere^3_{\tau(t)}$, is a $G$--invariant Willmore surface.
    \end{enumerate}
    We say that $\gamma$ is \emph{parametrizing the bifurcation} and the trace $\ima(\gamma)$ is a \emph{bifurcation} at $\tau$.
\end{definition}

\medskip
Now we move to the concrete cases. Let 
$$
x,y^\tau \colon \T^2 \to \berger
$$
be parametrizations of the Clifford torus and of the bifurcated brother, respectively, given by 
$$
x(z,w)=\tfrac{1}{\sqrt{2}}(z,w), \quad \text{and} \quad y^\tau(z,w)=(a(\tau)z,b(\tau)w),
$$
where the expressions of $a(\tau)$ and $b(\tau)$ were given in \cref{eq: constantes do irmao bifurcado}. For $\mu,\nu \in \R$, define the maps $$
R_{\mu,\nu}  ,\ \swap ,\ \conj \colon \berger \to \berger,
$$ 
by
$$
R_{\mu,\nu}(z,w) = (e^{i\mu}z, e^{i\nu}w),\quad \swap(z,w)=(w,z), \ \quad \text{and} \quad \conj(z,w) = (\overline{z},\overline{w}). 
$$
We set 
$$
K = \operatorname{span}\{R_{\theta,-\theta}, \conj : \theta \in \R\} \subset \iso(\berger).
$$
For every $\tau>0$, it is straightforward to verify that $K$ is a subgroup of both $\fix(\Sigma_\tau)$ and $\fix(\widetilde\Sigma_\tau)$. Also, for every $k \in K$, the maps $\varphi^{y,\tau}_k$ and $\varphi^x_k$ coincide for $y^\tau$ and $x$, being independent of $\tau$. Henceforth, we denote both maps by $\varphi_k$. More explicitly, 
\begin{align}\label{eq: K-invariant functions on T2}
    \varphi_{R_{\theta,-\theta}}(z,w) = (e^{\theta}z,e^{-\theta}w), \quad \text{and} \quad \varphi_{\conj}(z,w)=(\overline{z},\overline{w}).
\end{align}
for every $(z,w) \in \T^2$. Moreover, we may set
\begin{equation*}
    C^{l,\alpha}_K(\T^2) = C^{l,\alpha}_{x,K}(\T^2) = C^{l,\alpha}_{y,K}(\T^2).
\end{equation*}
Let 
\begin{align*}
  N(z,w)=\tfrac{1}{\sqrt{2}}(-z,w), \quad \widetilde{N}_\tau(z,w)= (-b(\tau)z,a(\tau)w), \quad (z,w) \in \T^2,
\end{align*} 
define the unit normal vector fields along $x$ and $y^\tau$, respectively, in the round sphere $\Sphere^3_1$. By \cref{eq: metrica de berger}, since $N$ and $\widetilde{N}_\tau$ are orthogonal to $\xi$ with respect to the round metric, they are also the unit normal vector fields along $x$ and $y^\tau$ in $\berger$. Moreover, by \cref{eq: conexão da Esfera de Berger}, the geodesics of $\berger$ with velocities parallel to $N$ or to $\widetilde{N}$ coincide with those of $\Sphere^3_1$. More explicitly,
\begin{align}
    x_f(p) &= \exp_{x(p)}^\tau(f(p)N(p)) = \exp_{x(p)}^1(f(p)N(p)), \label{eq: parametrização grafica do Clifford} \\
    y_f^\tau(p) &= \exp_{y^\tau(p)}^\tau(f(p)\widetilde{N}_\tau(p)) = \exp_{y^\tau(p)}^1(f(p)\widetilde{N}_\tau(p)). \label{eq: parametrização grafica do BB} 
\end{align}

Let $\tau_m$ be the unique positive solution to the equation $\lambda_{m,m}^+(\tau)=0$ (see \cref{remark: behavior of the sequence tau_m}).

\begin{theorem}[Bifurcations of the Clifford torus]\label{theorem: Bifurcations of the Clifford Torus}
   The sequence $\{\tau_m\}_{m \in \N_0}$ is precisely the set of all $K$--invariant bifurcation parameter of the Clifford Torus as a Willmore surface. Moreover, we may parametrize the bifurcation at $\tau_m$ by a smooth path
    \begin{equation*}
        \gamma_m \colon (- \delta_m, \delta_m) \to \mathcal{U}_m, \quad \gamma_m(t)=(\tau_m(t),f_m(t)),
    \end{equation*}
    in a neighbourhood $\mathcal{U}_m \subset (0,+\infty) \times C^{4,\alpha}_K(\T^2)$ of $(\tau_m,0)$ such that the following properties hold.
    \begin{enumerate} 
        \item[$(a)$] If $(\tau,f) \in \mathcal{U}_m$ yields a Willmore surface $x_f \colon \T^2 \to \berger$, then $f=0$ or $(\tau,f) \in \ima(\gamma_m)$.
        
        \item[$(b)$] The variation of $x$ given by 
        $$F_m \colon (-\delta_m,\delta_m) \times \T^2 \to \Sphere^3, \quad F_m(t,p)= x_{f_m(t)}(p),
        $$ 
        has variational vector field $$
        \left.\frac{\partial}{\partial t}\right|_{t=0}F_m \big(t,(e^{i\theta},e^{i\varphi})\big) = \cos\big(m(\theta+\varphi)\big)N(e^{i\theta},e^{i\varphi}).
        $$     
    \end{enumerate}
    Finally, for $m \geq 2$, the following extra property holds:
    \begin{enumerate}
      \item[$(c)$] For every $t\in(-\delta_m,\delta_m)$ with $t \neq 0$, there exists no isometry
      $\psi \in \iso(\Sphere^3_{\tau_m(t)})$
      such that $\psi(\operatorname{Im}(x_{f_m(t)}))$ is a Hopf torus.
      \end{enumerate}
\end{theorem}
\medskip

\begin{remark}\label{remark: trivial bifurcation at 1}

    There are two bifurcations that we can describe explicitly. They occur at $\tau_1=1$ and $\tau_0 = \sqrt{2}$. The bifurcation at $\tau_0$ was described in \cref{remark: trivial bifurcation at sqrt 2}. Now we describe the other one. Let $s \mapsto R_s \in \iso(\Sphere^3_1) = O(4)$ be the curve of isometries given by 
    
    \medskip
    \begin{align*}
        R_s = 
        \begin{pmatrix}
        \cos(s) & 0 & -\sin(s) & 0 \\
        0 & \cos(s) & 0 & \sin(s) \\
        \sin(s) & 0 & \cos(s) & 0 \\
        0 & -\sin(s) & 0 & \cos(s)
        \end{pmatrix}.
    \end{align*}
    
    \medskip
    Notice that $R_0(\Sigma_\tau) = \Sigma_\tau$, $R_{\frac{\pi}{4}}(\Sigma_\tau) \ne \Sigma_\tau$, and $R_s(\Sigma_\tau)$ is $K$--invariant for every $s \in \R$. Indeed, it suffices to check that
    $R_s^{-1}R_{\theta,-\theta}R_s = R_{\theta,-\theta}$ and 
    $R_s^{-1}\conj R_s = \conj$.

\end{remark}
 
Now we describe the strategy to prove the bifurcation of $\Sigma_\tau$ by new Willmore surfaces. The $\widetilde{\Sigma}_\tau$ case follows the same ideas. The subgroup $K \subset \fix(\Sigma_\tau)$ plays two important roles. First, to bifurcate $\Sigma_\tau$, we apply Crandall--Rabinowitz (see \cref{Crandall-Rabinowitz}) defining the map $(\tau,f) \mapsto \mathcal{P}(\tau,f) = \mathcal{E}_\willmore(x_f)$, where $\mathcal{E}_\willmore$ is the Euler--Lagrange operator in $\berger$. Notice that $\tau \mapsto (\tau,0)$ is a path of solutions to $\mathcal{P}(\tau,f)=0$ and any solution $(\tau,f)$ yields a Willmore surface $x_f \colon  \T^2 \to \berger$. If we choose $C^{4,\alpha}(\T^2)$ to be the domain of $\mathcal{P}(\tau, \cdot)$, the hypothesis (b) of Crandall--Rabinowitz is false, since the eigenvalues of $\mathcal{L}_{\Sigma_\tau}$ have even multiplicities by \cref{proposition: spectrum of the Clifford torus}. So we restrict the domain of $\mathcal{P}(\tau,\cdot)$ to a subspace where a certain eigenvalue is simple, namely, $\lambda_{m,m}^+$. The corresponding subspace is $C^{4,\alpha}_K(\T^2)$. The second role of $K$ is that, while the Crandall--Rabinowitz theorem is from the realm of Analysis and does not distinguish geometries, if $f$ is $K$--invariant, then $x_f$ is also $K$--invariant, and we can exploit this fact to prove item (c). 

\medskip
We begin with two lemmas whose proofs are postponed to \cref{app: lemma's proofs}.
\medskip

In order to describe $\fix(\Sigma_\tau)$ and $\fix(\widetilde{\Sigma}_\tau)$, we identify 
$$
\Sphere^1 \times \Sphere^1 = \{ R_{\mu,\nu} : \mu,\nu \in \R \} \quad \text{by} \quad 
(e^{i\mu},e^{i\nu}) \mapsto R_{\mu,\nu},
$$ 
and 
$$
\Z_2 \times \Z_2 = \{I,\swap,\conj, \swap \circ \conj \} \quad by \quad (1,0) \mapsto \swap, \quad \text{and} \quad(0,1) \mapsto \conj,
$$
where $I$ denotes the identity.

\begin{lemma}\label{lemma: Symetries of Clifford Torus}
Let $\Sigma = \Sphere^1(r_1) \times \Sphere^1(r_2)$ with $r_1 \neq r_2$ both positive, and assume $r_1^2 + r_2^2 = 1$. For $\tau \ne 1$, the groups of symmetries of the Clifford torus and of $\Sigma$ in the Berger sphere $\berger$ are given by 
\begin{align*}
\fix(\Sigma_\tau) &= \big\{R_{\mu,\nu} \circ T :  \mu,\nu \in \R, \ T \in  \{I, \swap, \conj, \swap \circ \conj\} \big\}, \\
\fix(\Sigma) &= \big\{R_{\mu,\nu} \circ T :  \mu,\nu \in \R, \ T \in  \{I, \conj \} \big\}.
\end{align*}
Moreover, 
$$
\Sphere^1 \times \Sphere^1 \rtimes_{\phi} \Z_2 \times \Z_2 \to \fix(\Sigma_\tau), \quad (R_{\mu,\nu}, T) \mapsto R_{\mu,\nu} \circ T,
$$
is a group isomorphism, where the action $$\Z_2 \times \Z_2 \overset{\phi}{\curvearrowright} \Sphere^1 \times \Sphere^1, \quad T \cdot_\phi R_{\mu,\nu} = T \circ R_{\mu,\nu} \circ T^{-1},$$ is given by conjugation.
\end{lemma}

Let 
$$
\big(\overline{C_K^{0,\alpha}(\T^2)}^{L^2},\ip{\cdot}{\cdot}_{L^2}\big)
$$
be the Hilbert space where $\overline{C_K^{0,\alpha}(\T^2)}^{L^2}$ is the $L^2$--closure of the H\"older space $C_K^{0,\alpha}(\T^2)$ and the $L^2$--inner product is induced by the canonical metric on the torus $\T^2 = \Sphere^1 \times \Sphere^1$. Similarly, let
$$
\big(\overline{C_K^{4,\alpha}(\T^2)}^{H^4},\ip{\cdot}{\cdot}_{H^4}\big)
$$
be the Hilbert space where $\overline{C_K^{4,\alpha}(\T^2)}^{H^4}$ is the $H^4$--closure of the H\"older space $C_K^{4,\alpha}(\T^2)$, and the $H^4$--inner product is defined by 
\begin{align*}
    \forall u,v \in H^4(\T^2), \quad \ip{u}{v}_{H^4} := \int_{\T^2} (I-\Delta_{can})^2u \, (I-\Delta_{can})^2v \, d\T^2,
\end{align*}
where $\Delta_{can}$ and $d\T^2$ are the Laplace--Beltrami operator and the volume form, respectively, corresponding to the canonical metric on $\T^2$. Define 
$$
C_{m,n}^\pm,S_{m,n}^\pm \colon \T^2 \to \R
$$ 
by
$$
C_{m,n}^\pm(e^{i\theta},e^{i\varphi}) = \cos(m\theta \pm n\varphi), \ S_{m,n}^\pm(e^{i\theta},e^{i\varphi}) = \sin(m\theta \pm n\varphi).
$$
It is well known that there is a family 
$$
\mathcal{B} \subset \{C_{m,n}^\pm,S_{m,n}^\pm\}_{m,n \in \N_0}
$$ which is an orthogonal (non-normalized) basis of the Hilbert spaces $\big(L^2(\T^2),\ip{\cdot}{\cdot}_{L^2}\big)$ and $\big(H^4(\T^2),\ip{\cdot}{\cdot}_{H^4}\big)$. To obtain a basis, one has to discard $S_{0,m}^+$ or $S_{0,m}^- \,$, for example.

\begin{lemma}\label{lemma: Orthogonal basis of K-invariant functions}
   The family $\{C_{m,m}^+\}_{m \in \N_0}$ is an orthogonal (non--normalized) basis of both Hilbert spaces $\big(\overline{C_K^{0,\alpha}(\T^2)}^{L^2},\ip{\cdot}{\cdot}_{L^2}\big)$ and $\big(\overline{C_K^{4,\alpha}(\T^2)}^{H^4},\ip{\cdot}{\cdot}_{H^4}\big)$.
\end{lemma}

Now we are ready to prove our main theorem.

\begin{proof}[Proof of \cref{theorem: Bifurcations of the Clifford Torus}]

Let $U \subset C^{4,\alpha}_K(\T^2)$ be a sufficiently small neighbourhood of the origin such that, for every $f \in U$, we have that $x_f$ is an embedding ($x_f$ do not depend of $\tau$ by \cref{eq: parametrização grafica do Clifford}). In the light of \cref{euler-lagrange equation of willmore surfaces in the Berger sphere}, define
$$
\mathcal{P} \colon (0,+\infty) \times U \to C^{0,\alpha}_K(\T^2)
$$
by
$$
\mathcal{P}(\tau,f) = -\Delta^\tau_f H^\tau_f - \big(|\traceless{A}^\tau_f|+4(1-\tau^2)(1+(C^\tau_f)^2)\big)H^\tau_f + 16(1-\tau^2)A^\tau_f(T^\tau_f,A^\tau_f),
$$
where $\Delta^\tau_f, H^\tau_f, A^\tau_f, C^\tau_f, T^\tau_f$ are all computed with respect to the pull-back metric $x_f^*(g_\tau)$. Hence, $\mathcal{P}(\tau,f)$ is the first variation operator of the Willmore energy applied to $x_f \colon \T^2 \to \berger$ and $x_f \colon \T^2 \to \berger$ is a Willmore surface if, and only if, $\mathcal{P}(\tau,f) = 0$. The codomain of $\mathcal{P}$ is well--defined because of the \cref{claim: K-invariance of Euler-Lagrage equation} below. Since the Clifford torus is a Willmore surface in every Berger sphere,
\begin{equation*}
    \forall \tau>0, \quad \mathcal{P}(\tau,0) = 0.
\end{equation*}
We check the hypothesis of \cref{Crandall-Rabinowitz} for a bifurcation of the path $\tau \mapsto (\tau,0)$, as roots of $\mathcal{P}$, at $\tau = \tau_m$ for a fixed $m \in \N_0$. 

\begin{claim}\label{claim: K-invariance of Euler-Lagrage equation}
The first variation operator of the Willmore functional applied to $x_f$ is $K$--invariant. More precisely, for every $f \in U$ and every $k \in K$, it holds that 
\begin{equation*}
    \mathcal{E}_\willmore(x_f) \circ \varphi_k = \mathcal{E}_\willmore(x_f).
\end{equation*}  
\end{claim}
\begin{proof}
    For any given immersion $z \colon \T^2 \to \berger$, let $N_z,A_z, H_z,T_z,C_z,\Delta_z$ and $\nabla^z$ denote the unit normal vector field, second fundamental form, mean curvature, tangent projection of $\xi$, angle function, Laplace--Beltrami operator, and connection along $z$, respectively, corresponding to the pull-back metric $z^*g$. Let $f \in U$ and $k\in K$. Since 
    \begin{equation*}
        k \circ x_f = x_f \circ \varphi_k,
    \end{equation*}
    we have that
    \begin{equation*}
        \varphi_k \colon (\T^2,x_f^*g) \to (\T^2,x_f^*g)
    \end{equation*}
    is an isometry. Fix $p \in \T^2$ and let $\{e_i\}$ be a $(x_f^*g)$--orthonormal basis of $T_p\T^2$. For the computations below, we use that $\{d\varphi_ke_i\}$ is an $(x_f^*g)$--orthonormal basis of $T_{\varphi_k(p)}\T^2$.

    \medskip
    \noindent
    $H_{x_f}$ is $K$-invariant:   
    \begin{align*}
        H_{x_f \circ \varphi_k}(p) &= A_{x_f \circ \varphi_k}(e_i,e_i) = g(\nabla^{x_f \circ \varphi_k}_{e_i}N_{x_f \circ \varphi_k}, d(x_f \circ \varphi_k)e_i) \\ 
        &= g\big(\nabla^{x_f \circ \varphi_k}_{e_i}(N_{x_f} \circ \varphi_k),dx_f(d\varphi_ke_i)\big) = g\big(\nabla^{x_f}_{d\varphi_ke_i}N_{x_f},dx_f(d\varphi_ke_i)\big) \\
        &=H_{x_f}(\varphi_k(p)),
    \end{align*}
    where we omitted the sum over $i$. Therefore,
    \begin{equation*}
       H_{x_f} \circ \varphi_k = H_{x_f \circ \varphi_k} .
    \end{equation*}
    On the other hand, since $k$ is a symmetry that preserves orientation,
    \begin{equation*}
       H_{x_f \circ \varphi_k} = H_{k \circ x_f} = H_{x_f} .
    \end{equation*}
    $|\traceless{A}_{x_f}|$ is $K$-invariant: since $|\traceless{A}_{x_f}|^2 = |A_{x_f}|^2 - \frac{1}{2}(H_{x_f})^2$ and $H_{x_f}$ is $K$--invariant, we only need to check this property for $|A_{x_f}|^2$. Indeed,
    \begin{align*}
        \big|A_{x_f}\big(\varphi_k(p)\big)\big|^2 &= g\big(\nabla_{d\varphi_k e_i}^{x_f} N_{x_f}, dx_f(d\varphi_k \, e_i)\big) = g\big(\nabla_{d\varphi_k e_i}^{x_f} N_{x_f}, dk(dx_f \, e_i)\big)\\  
        & = g(dk^{-1} \, \nabla_{d\varphi_ke_i}^{x_f} N_{x_f},dx_f \, e_i) = g\big(\nabla_{d\varphi_k e_i}^{k^{-1} \circ x_f} (dk^{-1} \, N_{x_f}),dx_f \, e_i\big) \\
        & = g\big(\nabla^{x_f}_{d\varphi_{k^{-1}} d\varphi_k e_i} N_{x_f}, dx_f \, e_i \big) = |A_{x_f}(p)|^2.
    \end{align*}
    $\Delta_{x_f}H_{x_f}$ is $K$--invariant: since $\varphi_k$ is an isometry and $H_{x_f}$ is $K$--invariant,
    \begin{align*}
        \Delta_{x_f} H_{x_f} = \Delta_{x_f} (H_{x_f} \circ \varphi_k) = (\Delta_{x_f}H_{x_f}) \circ \varphi_k. 
    \end{align*}
    $(C_{x_f})^2$ is $K$--invariant: since $k$ is a symmetry that preserves orientation,
    \begin{align*}
      dk \, N_{x_f} =N_{k \circ x_f} = N_{x_f \circ \varphi_k} = N_{x_f} \circ \varphi_k.    
    \end{align*}
    We also observe that 
    \begin{align*}
        dk(\xi) = \begin{cases}
            \xi \circ k, &\text{if $k \in U^+$ }, \\
            -\xi \circ k, &\text{if $k \in U^{-}$}.
        \end{cases}
    \end{align*}
    Using the two equations above, we obtain 
    \begin{align*}
        C_{x_f} \circ \varphi_k &= g(N_{x_f} \circ \varphi_k, \xi \circ x_f \circ \varphi_k) = g(dk \, N_{x_f},\xi \circ k \circ x_f) \\
        &= \pm \, g(dk \, N_{x_f},dk\big(\xi \circ x_f )\big) = \pm \, g(N_{x_f}, \xi \circ x_f) = \pm \, C_{x_f}. 
    \end{align*}
    $A_{x_f}(T_{x_f},T_{x_f})$ is $K$--invariant: only in this part, denote the tangential projection onto $dx_f(T\T^2)$ by $( \cdot)^\top$ and let $W \in \campos(\T^2)$ be the vector field defined by $dx_f \, W = T_{x_f}$. Note that
    \begin{align*}
        T_{x_f} \circ \varphi_k &= (\xi \circ x_f \circ \varphi_k)^\top = (\xi \circ k \circ x_f)^\top \\
        &= \pm \, \big(dk( \xi \circ  x_f)\big)^\top = \pm \, dk \, T_{x_f}.
    \end{align*}
    From the definition of $W$ and the injectivity of $dx_f$, we obtain that for every $\tilde{k} \in K$, $W$ is $(\varphi_{\tilde{k}})$--related to itself,
    \begin{align*}
        d\varphi_{\tilde{k}} \, W = W \circ \varphi_{\tilde{k}}.
    \end{align*}
    From the two equations above, it follows that
    \begin{align*}
        A_{x_f}(T_{x_f},T_{x_f}) \circ \varphi_k &= g(\nabla^{x_f}_{W \circ \varphi_k} N_{x_f}, dk \, T_{x_f}) = g(dk^{-1} \, \nabla^{x_f}_{W \circ \varphi_k} N_{x_f}, T_{x_f}) \\
        &= g\big(\nabla_{W \circ \varphi_k}^{k^{-1} \circ x_f}(dk^{-1} N_{x_f}), T_{x_f}\big) = g\big(\nabla_{W \circ \varphi_k}^{x_f \circ \varphi_{k^{-1}}}(N_{x_f} \circ \varphi_{k^{-1}}),T_{x_f}\big) \\
        &= g(\nabla_{d\varphi_{k^{-1}}(W \circ \varphi_k)}^{x_f} N_{x_f}, T_{x_f}) = g(\nabla_W^{x_f} N_{x_f}, T_{x_f}) = A_{x_f}(T_{x_f},T_{x_f}).
    \end{align*}
    By the computations above and \cref{euler-lagrange equation of willmore surfaces in the Berger sphere}, we are done.
    
\end{proof}

\begin{claim}\label{claim:smoothness_of_P}
    The map $\mathcal{P} \colon (0,+\infty) \times U \to C^{0,\alpha}_K(\T^2)$ is smooth.
\end{claim}
\begin{proof}
    It is enough to prove that $\mathcal{P}$ is smooth when regarded as a map $(0,+\infty) \times V \to C^{0,\alpha}(\T^2)$ for some neighborhood $V \subset C^{4,\alpha}(\T^2)$ of the origin such that $x_f$ is an embedding for each $f \in V$. Endow $\T^2$ with its canonical product metric. Since there is an isometry $\zeta:\T^2 \to \R^2/(\pi\sqrt{2}\Z)^2$, $\mathcal{P}$ induces a map 
    $$
    \mathcal{P}_1 \colon (0,+\infty) \times U_1 \to C^{0,\alpha}\big(\R^2/(\pi\sqrt{2}\Z)^2\big),
    $$ 
    where $U_1 \subset C^{4,\alpha}(\R^2/(\pi\sqrt{2}\Z)^2)$, such that $\mathcal{P}$ is smooth if and only if $\mathcal{P}_1$ is smooth. Let $C^{l,\alpha}_{\pi\sqrt{2}}(\R^2)$ be the subspace of the functions $ f \in C^{l,\alpha}(\R^2)$ that are $\pi\sqrt{2}$ periodic in both of its entries. There is a natural Banach isomorphism $C^{l,\alpha}\big(\R^2/(\pi\sqrt{2}\Z)^2\big) \to C^{l,\alpha}_{\pi\sqrt{2}}(\R^2)$. Thus, $\mathcal{P}_1$ induces a map 
    $$
    \mathcal{P}_2 \colon (0,+\infty) \times U_2 \to C^{0,\alpha}_{\pi\sqrt{2}}(\R^2),
    $$ 
    where $U_2 \subset C^{4,\alpha}_{\pi\sqrt{2}}(\R^2)$, such that $\mathcal{P}_1$ is smooth if and only if $\mathcal{P}_2$ is smooth. The map $\mathcal{P}_2$ is of the form $\mathcal{P}_2(\tau,f) = h(\tau,f,\partial_1f,\partial_2f,\partial_{11}f,\ldots,\partial_{2222}f)$, where $h \colon (0,+\infty) \times \R^{31} \to \R$ is smooth, since it depends only on $(\Sigma_\tau,g_\tau)$, the parametrization $x$, and the choice of isometry $\zeta$. The explicit expression of $h$ does not matter. Define the map
    $$
    J \colon (0,+\infty) \times C^{4,\alpha}_{\pi\sqrt{2}}(\R^2) \to (0,\infty) \times C^{0,\alpha}(\R^2,\R^{31})
    $$
    by
    $$
    J(\tau,f) = (\tau,f,\partial_1f,\partial_2f,\partial_{11}f, \ldots ,\partial_{2222}f),
    $$
    and the composition map
    $$
    H \colon (0,+\infty) \times C^{0,\alpha}(\R^2,\R^{31}) \to C^{0,\alpha}(\R^2)
    $$
    by
    $$
    \forall q \in \R^2, \quad  H(\tau,w)(q) = h(\tau,w(q)). 
    $$
    Note that $J$ is a restriction of a bounded linear map. The smoothness of $H$ follows by a straightforward adaptation of the proof of the Proposition in the appendix of \cite{Brian1987}. Thus, $\mathcal{P}_2 = H \circ J$ is smooth, which implies the smoothness of $\mathcal{P}$.

\end{proof}

\medskip

\begin{claim}\label{claim:Fredholmness_of_L}
    Let $L = D_0(\mathcal{P}(\tau_m, \cdot)) \colon C^{4,\alpha}_K(\T^2) \to C^{0,\alpha}_K(\T^2)$ be the differential of $\mathcal{P}(\tau_m, \cdot)$ at $0 \in C^{4,\alpha}_K(\T^2)$. Then, $L$ is a Fredholm operator such that:
    \begin{enumerate}[label=\textbullet]
        \item The Fredholm index of $L$ is 0.
        \item $\operatorname{Ker}(L) = \R C_{m,m}^+$ and $\frac{C^{0,\alpha}_K(\T^2)}{\operatorname{Im}(L)}$ is generated by the class of $C_{m,m}^+$.
    \end{enumerate}
\end{claim}

\begin{proof}
    From \cref{Clifford torus second variation operator}, we have that
\begin{equation*}
    \forall u \in C^{4,\alpha}_K(\T^2), \quad L(u)=-(\Delta^\tau_f)^2 u - 2(4-\tau^2)\Delta^\tau_f u - 8(2-\tau^2)u + 16(1-\tau^2)L_\xi^2 (u).
\end{equation*}
Consider also the extension of $L$ to the Sobolev closure,
namely the operator $$\overline{L} \colon \ \overline{C^{4,\alpha}_K(\T^2)}^{H^4} \longrightarrow \ \ \overline{C^{0,\alpha}_K(\T^2)}^{L^2}$$ defined by the same formula as above. From \cref{proposition: spectrum of the Clifford torus} and (\ref{Clifford bifurcations parameters}), we obtain
$$
\forall n \in \N_0, \quad \overline{L}(C_{n,n}^+) + \lambda_{n,n}^+(\tau_m)C_{n,n}^+ = 0,
$$
and $\overline{L}(C_{n,n}^+) = 0$ if and only if $n = m$. From \cref{lemma: Orthogonal basis of K-invariant functions}, it follows that $\overline{L}$ is a Fredholm operator with Fredholm index 0 and $\ker{\overline{L}} = \R C_{m,m}^+$. Therefore, $\ker{L} = \R C_{m,m}^+$, and it only remains to verify that the class of $C_{m,m}^+$ generates $\frac{C^{0,\alpha}_K(\T^2)}{\operatorname{Im}(L)}$. Let 
\begin{equation*}
    P \colon C^{0,\alpha}_K(\T^2) \longrightarrow \frac{C^{0,\alpha}_K(\T^2)}{\operatorname{Im}(L)}, \quad Q\colon \overline{C^{0,\alpha}_K(\T^2)}^{L^2} \longrightarrow \frac{\overline{C^{0,\alpha}_K(\T^2)}^{L^2}}{\operatorname{Im}(\overline{L})},
\end{equation*}
denote the natural projections onto the quotient spaces. Let $v \in C^{0,\alpha}_K(\T^2)$. Note that
$$
Q(v) \in \frac{\overline{C^{0,\alpha}_K(\T^2)}^{L^2}}{\operatorname{Im}(\overline{L})} = \R Q(C_{m,m}^+).
$$
Thus, there exist $r \in \R$ and $u \in \overline{C_K^{4,\alpha}(\T^2)}^{H^4}$ such that 
$$
\overline{L}(u) = rC_{m,m}^+ - v \in C^{0,\alpha}_K(\T^2).
$$
Since $\overline{L}$ is a linear, elliptic differential operator of order 4, by the regularity statement in \cite[Th. 31, p. 464]{BesseEinsteinManifolds} (for a proof, see \cite[Th. 6.2.5, p. 236]{MorreyMultipleIntegrals}), we have that $u \in C^{4,\alpha}_K(\T^2)$. Hence, $\overline{L}(u) = L(u) \in \operatorname{Im}(L)$ and $P(v) = rP(C_{m,m}^+)$.

\end{proof}

\begin{claim}\label{claim:condition_on_second_derivative_of_L}
   $\left.\frac{d}{d\tau}\right|_{\tau=\tau_m} D_0(\mathcal{P}(\tau, \cdot))(C_{m,m}^+) \notin \ima(L)$.
\end{claim}

\begin{proof}

    By \cref{Clifford torus second variation operator,claim:Fredholmness_of_L,proposition: spectrum of the Clifford torus}, we obtain
    \begin{align*}
        \left.\frac{d}{d\tau}\right|_{\tau=\tau_m} D_0(\mathcal{P}(\tau, \cdot))(C_{m,m}^+) 
        &= \left.\frac{d}{d\tau}\right|_{\tau=\tau_m} \mathcal{L}_{\Sigma_\tau}(C_{m,m}^+) \\ 
        &= \left.\frac{d}{d\tau}\right|_{\tau=\tau_m} \lambda_{m,m}^+(\tau) C_{m,m}^+ \\
        &= \left.\frac{d}{d\tau}\right|_{\tau=\tau_m} \big(\frac{16m^4}{\tau^4} + \frac{32m^2}{\tau^2} + 16 -56m^2 - 8\tau^2\big)C_{m,m}^+ \\
        &= \underbrace{\big(-\frac{64m^4}{\tau_m^5} - \frac{64m^2}{\tau_m^3} -16\tau_m\big)}_{ \ne \ 0}C_{m,m}^+ \notin \ima(L). 
    \end{align*}
    
\end{proof}

By \cref{claim:smoothness_of_P,claim:Fredholmness_of_L,claim:condition_on_second_derivative_of_L}, we may apply the Crandall--Rabinowitz \cref{Crandall-Rabinowitz} to bifurcate the path $\tau \mapsto (\tau,0)$, as roots of $\mathcal{P}$, at $\tau=\tau_m$. More precisely, there exist $\delta_m > 0$, an open neighbourhood
$$
\mathcal{U}_m \subset (0,+\infty) \times C_K^{4,\alpha}(\T^2)
$$
of $(\tau_m,0)$, and a smooth path 
$$
\gamma_m \colon (-\delta_m,\delta_m) \to \mathcal{U}_m, \quad \gamma_m(t)=(\tau_m(t),f_m(t)),
$$ 
such that 
\begin{enumerate}
    \item[(i)] $\tau_m(0) = \tau_m$, $f_m(0) = 0 \in C^{4,\alpha}_K(\T^2)$, $f_m(t) \neq 0$ when $t \ne 0$, and $f_m'(0)=C_{m,m}^+$,
    \item[(ii)] $\{\mathcal{P}(\tau,f) = 0\} \cap \mathcal{U}_m = \big(\{(\tau,0) \mid \tau>0 \} \cup \ima(\gamma_m)\big) \cap \mathcal{U}_m$.
\end{enumerate}
From $(i)$ and $(ii)$, it follows that $\tau_m$ is a $K$--invariant bifurcation parameter for $\Sigma_\tau$ as a Willmore surface such that properties $(a)$ and $(b)$ of \cref{theorem: Bifurcations of the Clifford Torus} hold. By the \cref{remark: behavior of the sequence tau_m,lemma: Orthogonal basis of K-invariant functions}, if $\tau \ne \tau_m$ for every $m \in \N_0$, then 
$$
D_0\big(\mathcal{P}(\tau, \cdot )\big) \colon C^{4,\alpha}_K(\T^2) \to C^{0,\alpha}_K(\T^2)
$$
is a linear isomorphism. Therefore, by the implicit function theorem, the set all $K$--invariant bifurcation parameter of the Clifford Torus as a Willmore surface is exactly $\{\tau_m\}_{m \in \N_0}$.

\medskip
It only remains to verify property $(c)$: Suppose that there exist $m \ge 2$, $t \in (-\delta_m,\delta_m)$, $\psi \in \iso(\Sphere^3_{\tau_m(t)}) = U^+ \cup U^-$, and a Hopf torus $\hopftori \subset \Sphere^3_{\tau_m(t)}$ such that $\ima(x_{f_m(t)}) = \psi(\hopftori)$. We will show that $t=0$.

\medskip
Since the functions 
\begin{align*}
    s \mapsto \willmore\big(\ima(x_{f_m(s)})\big), \quad s \mapsto \willmore(\widetilde{\Sigma}_{\tau_m(s)}), \quad \text{and} \quad s \mapsto \tau_m(s),
\end{align*}
are continuous, with $\tau_m(0) = \tau_m$ and $\ima(x_{f_m(0)}) = \Sigma_\tau$, by \cref{remark: behavior of the sequence tau_m,remark: Willmore energy of Hopf tori family} we may decrease $\delta_m$ to ensure that 
$$
\begin{cases}
    \tau_m(t) < \sqrt{2}, \quad &\text{if} \quad m \in \{2,3\}, \\
    \willmore\big(\ima(x_{f_m(t)})\big)>\willmore(\widetilde{\Sigma}_{\tau_m(t)}), \quad &\text{if} \quad m \geq 4,
\end{cases}
$$
From the fact that $f_m(t)$ is $K$--invariant, we have 
$$
\psi^{-1} \{R_{\theta,-\theta}\} \psi \subset \iso_0(\hopftori),
$$ 
where $\iso_0(\hopftori)$ is the connected component of the identity of $\iso(\hopftori)$. Indeed,
\begin{align}\label{eq. rotações fixando imagem isometrica do toro de hopf}
    \{R_{\theta,-\theta}\} \subset \fix\big(\ima(x_{f_m(t)})\big) &= \fix(\psi(\hopftori)) = \psi\fix(\hopftori)\psi^{-1}.
\end{align}
From the fact that $\hopftori = \pi^{-1}(\ima(\beta))$, we also have 
$$
\{R_{\varphi,\varphi}\} \subset \iso_0(\hopftori).
$$ 

\medskip
Denote $H_1 = \{R_{\varphi,\varphi}\}$ and $H_2 = \psi^{-1} \{R_{\theta,-\theta}\}\psi$. Notice that $H_1$ and $H_2$ are compact one–
dimensional Lie subgroups of $\iso_0(\hopftori)$. Also, $H_1 \ne H_2$ because 
$$
H_1 = \{ \psi \in U(2) \mid \psi  \zeta = \zeta \psi \quad \forall \zeta \in U(2) \}
$$
is the centre of $U(2)$. Since $\hopftori$ is a flat torus, there exists a group isomorphism 
$$
\iso_0(\hopftori) \to \R^2/\Z^2.
$$  
Therefore,
$$
H_3:=\operatorname{span}\left\{H_1,H_2\right\} = \iso_0(\hopftori) \subset U^+ = U(2),
$$ 
and $H_3$ is a maximal torus of $U^+$. By the Maximal Torus Theorem (see e.g. \cite[Th. 11.9, p. 316]{Hall2015}) there exists $\zeta \in U^+$ such that 
$$
H_3  = \zeta^{-1} \{R_{\mu,\nu}\}\zeta.
$$
Since $H_3 \subset \fix(\hopftori)$, we obtain
$$
\{R_{\mu,\nu}\} \subset \zeta \fix(\hopftori)\zeta^{-1} = \fix\big(\zeta(\hopftori)\big).
$$
    
\begin{claim}\label{claim: Rigidity of Symmetries of Hopf tori}
    Let $\Sigma \subset \berger$ be a connected surface. If $\{R_{\mu,\nu}\}_{\mu,\nu \in \R} \subset \fix (\Sigma)$, then there exists a curve $\gamma$ parametrizing a parallel circle $\Sphere^2\left(\frac{1}{2}\right) \cap (\C \times \{c_0\})$, at some height $c_0 \in (-\frac{1}{2},\frac{1}{2})$, such that $\Sigma = \Sigma^{\gamma}$ is a Hopf Torus. Furthermore, if $\Sigma$ is also a Willmore surface, then
    \begin{align*}
        \tau \in (0,\sqrt{2}] \ &\Rightarrow \ \Sigma = \Sigma_\tau. \\
        \tau \in (\sqrt{2},\infty) \ &\Rightarrow \ \Sigma =\Sigma_\tau \ \ \text{or} \ \ \Sigma = \widetilde{\Sigma}_\tau.
    \end{align*}   
\end{claim}

\begin{proof}
Fix $(z_0,w_0) \in \Sigma$. From the hypothesis, it follows that
$$
\widetilde \Sigma := \{(e^{i\mu}z_0,e^{i\nu}w_0) \in \berger : \mu,\nu \in \R \} \subseteq \Sigma.
$$
Since 
$$
\T^2 \to \Sigma, (e^{i\mu},e^{i\nu}) \mapsto (e^{i\mu}z_0,e^{i\nu}w_0),
$$ is an embedding with image $\widetilde \Sigma$, we see that $\widetilde \Sigma$ is open and closed in $\Sigma$. By connectivity, we obtain $\widetilde \Sigma = \Sigma$. It follows from the definition of $\pi$, see (\ref{eq: Hopf submersion}), that $\pi(\widetilde{\Sigma}) = \Sphere^2\left(\frac{1}{2}\right) \cap (\C \times \{c_0\})$, where $c_0 = \frac{1}{2}(|z_0|^2 - |w_0|^2)$.

\medskip
Let $\gamma$ parametrize the parallel circle at height $c_0$ which has constant geodesic curvature in $\Sphere^2\left(\frac{1}{2}\right)$, so that $\Sigma = \Sigma^\gamma$. Supposing that $\Sigma$ is a Willmore surface, by \cref{corollary: Classification of Hopf Willmore surfaces,remark: trivial bifurcation at sqrt 2}, we conclude.

\end{proof}

By \cref{claim: Rigidity of Symmetries of Hopf tori}, there exists a smooth closed regular curve $\gamma$ parametrizing some horizontal circle in $\Sphere^2\left(\frac{1}{2}\right)$, such that $\zeta(\hopftori) = \Sigma^\gamma$. Since $\zeta$ and $\psi$ are isometries of $\Sphere^3_{\tau_m(t)}$, we obtain that $\Sigma^\gamma$ is a Willmore surface. Since 
$$
\begin{cases}
    \tau_m(t) < \sqrt{2}, \quad &\text{if} \quad m \in \{2,3\}, \\
    \willmore(\Sigma^\gamma)=\willmore\big(\ima(x_{f_m(t)})\big)>\willmore(\widetilde{\Sigma}_{\tau_m(t)}), \quad &\text{if} \quad m \geq 4,
\end{cases}
$$
also by \cref{claim: Rigidity of Symmetries of Hopf tori}, we conclude that $\Sigma^\gamma = \Sigma_{\tau_m(t)}$.
Thus,
\begin{equation*}
    \{R_{\theta,-\theta}\} \subset \fix\big(\psi(\hopftori)\big) = \fix\big((\psi \circ \zeta^{-1})(\Sigma^\gamma)\big) = \fix\big((\psi \circ \zeta^{-1})(\Sigma_{\tau_m(t)})\big). 
\end{equation*}

    \begin{claim}\label{claim: Rigidity of symmetries of clifford torus}
    Let $\tau \neq 1$ and $\rho \in \iso(\berger)$. If $\{R_{\theta, -\theta}\}_{\theta \in \R} \subset \fix(\rho(\Sigma_\tau))$, then $\rho \in \fix(\Sigma_\tau)$.
    \end{claim}

    \begin{proof} 
    Suppose that $\{R_{\theta, -\theta}\}_{\theta \in \R} \subset \fix(\rho(\Sigma_\tau))$. From 
    $
    \fix(\rho(\Sigma_\tau))=\rho\fix(\Sigma_\tau)\rho^{-1},
    $
    we have
    $$
    \rho^{-1} \{R_{\theta, -\theta}\}_{\theta \in \R} \ \rho \subset \fix(\Sigma_\tau). 
    $$
    By \cref{lemma: Symetries of Clifford Torus}, 
    $$
    \fix(\Sigma_\tau) =  \bigcup_T \{R_{\mu,\nu} \circ T \mid \mu,\nu \in \R \},
    $$
    where $T$ runs over $\{I, \swap, \conj, \swap \circ \conj\}$. In particular, $\fix(\Sigma_\tau)$ has four connected components. Set
    $$
     \phi \colon \R \to \fix(\Sigma_\tau), \quad \phi(\theta) = \rho^{-1} R_{\theta,-\theta} \ \rho.
    $$ 
    Since $\phi$ is continuous and $\phi(0) = I$, we get $\ima(\phi) \subset \{R_{\mu,\nu}\}$.

    \medskip
    \noindent
    \emph{Case: $\rho \in U^+$.}
    Write $\rho = \begin{pmatrix}
        a & b \\
        c & d
    \end{pmatrix}$ satisfying
    \begin{align*}
    \begin{cases}
       |a|^2 + |b|^2 =|c|^2 + |d|^2 = 1, \\
       a\overline{c} + b\overline{d} = 0.
    \end{cases}
    \end{align*}
    We compute
    \begin{align}
        \phi(\theta) = \rho^{-1} R_{\theta,-\theta}\rho = 
        \begin{pmatrix}
            |a|^2 e^{i\theta} + |c|^2e^{-i\theta} & \overline{a}be^{i\theta} + \overline{c}de^{-i\theta} \\
            a\overline{b}e^{i\theta} + c\overline{d}e^{-i\theta} & |b|^2e^{i\theta}+|d|^2e^{-i\theta}
        \end{pmatrix} \label{algebraic identity of Fix Clifford in U(2)}.
    \end{align}
    Choosing $\theta = 0$, then $\theta = \frac{\pi}{2}$, and looking at the second row and first column of \cref{algebraic identity of Fix Clifford in U(2)} we obtain that either $a=0$ or $b=0$.
    Thus, there exist $\mu,\nu \in \R$ such that either $\rho = R_{\mu,\nu}$ or $\rho = R_{\mu,\nu} \circ \swap$.
    Hence, $\rho \in \fix(\Sigma_\tau)$.

    \medskip
    \noindent
    \emph{Case: $\rho \in U^-$.} 
    We have $\rho = \conj \circ \ \tilde{\rho}$ with $\tilde{\rho} \in U^+$. Since 
    \begin{align*}
        \phi(\theta) 
        &= \rho^{-1} R_{\theta,-\theta} \ \rho \\
        &= (\tilde{\rho})^{-1} \conj R_{\theta,-\theta} \conj\tilde{\rho} \\  
        &= (\tilde{\rho})^{-1} R_{-\theta,\theta} \ \tilde{\rho} \in \fix(\Sigma_\tau),
    \end{align*}
    the previous case implies that $\tilde{\rho} \in \fix(\Sigma_\tau)$, and hence $\rho \in \fix(\Sigma_\tau)$.
    
\end{proof}

We decrease $\delta_m$ again, if necessary, to ensure $\tau_m(s) \neq 1$ for every $s \in (-\delta_m,\delta_m)$. By \cref{claim: Rigidity of symmetries of clifford torus}, we obtain $\psi \circ \zeta^{-1} \in \fix(\Sigma_{\tau_m(t)})$ and 
\begin{equation*}
    \psi(\hopftori) = \psi\big(\zeta^{-1}(\Sigma_{\tau_m(t)})\big) = \Sigma_{\tau_m(t)}.
\end{equation*}
Hence,
\begin{equation*}
    \ima(x_{f_m(t)}) = \psi(\hopftori) = \Sigma_{\tau_m(t)}.
\end{equation*} 

\medskip
Since $\tau_m$ is a $K$--invariant bifurcation parameter of $\Sigma_\tau$, we obtain that $f_m(t) = 0$ and $t = 0$ (see \cref{definition: invariant bifurcation}).

\end{proof}

The argument based on the implicit function theorem, which shows that there are no other $K$--invariant bifurcation parameters for the Clifford torus, actually yields the following stronger result. 

\begin{corollary}[Local rigidity of the Clifford torus as a Willmore surface]
Let $\tilde{\tau}$ be a nondegenerate $K$--invariant Berger parameter for the Clifford torus as a Willmore surface, that is,
\begin{equation*}
    \tilde{\tau} \in (0,+\infty)\setminus\{\tau_m\}_{m\in\N_0}.
\end{equation*}
Then, there exist $\delta=\delta(\tilde{\tau})>0$ and an open neighbourhood of the origin
\begin{equation*}
    U=U_{\tilde{\tau}}\subset C^{4,\alpha}_K(\T^2),
\end{equation*}
such that if $f\in U$ and $x_f:\T^2\to\berger$ is a Willmore surface for some
\begin{equation*}
    \tau\in(\tilde{\tau}-\delta,\tilde{\tau}+\delta),
\end{equation*}
then $f=0$, and hence $x_f=x$ is the parametrization of the Clifford torus.
\end{corollary}

\begin{theorem}[Bifurcations of the bifurcated brother]\label{theorem: Bifurcations of the bifurcated brother}
The set of $K$--invariant parameters of the bifurcated brother as a Willmore surface is precisely $\{\tilde{\tau}_2,\tilde{\tau}_3\}$, where $\tilde{\tau}_m$ is the unique positive solution to the equation $\lambda^+_{m,m}(\tau) = 0$. Moreover, we may parametrize the bifurcation at $\tilde{\tau}_m$ by a smooth path
$$
\gamma_m \colon (- \delta_m, \delta_m) \to \mathcal{U}_m, \quad \gamma_m(t)=(\tau_m(t),f_m(t)),
$$
in a neighbourhood  $\mathcal{U}_m \subset (0,+\infty) \times C^{4,\alpha}_K(\T^2)$ of $(\tilde{\tau}_m,0)$ such that the following hold.
\begin{enumerate} 
    \item[$(a)$] If $(\tau,f) \in \mathcal{U}_m$ yields a Willmore surface $y^\tau_f \colon \T^2 \to \berger$, then $f=0$ or $(\tau,f) \in \ima(\gamma_m)$.
    \item[$(b)$] The variation of $y^{\tau_m}$ given by 
       $$
       F_m \colon (-\delta_m,\delta_m) \times \T^2 \to \Sphere^3, \quad F_m(t,p)= y^{\tau_m(t)}_{f_m(t)}(p),
       $$ 
       has variational vector field 
       $$
       \left.\frac{\partial}{\partial t}\right|_{t=0}F_m (t,p) =  \big(\phi'(\tilde{\tau}_m)\tau_m'(0) + C_{m,m}^+(p)\big)\widetilde{N}_{\tilde{\tau}_m}(p).
       $$
       where $\phi(\tau)=\frac{1}{2}\arccos\Big(\frac{\sqrt{\tau^2-2}}{\sqrt{\tau^2-1}}\Big)$.
    \item[$(c)$] For every $t\in(-\delta_m,\delta_m)$ with $t \neq 0$, there exists no isometry $\psi \in \iso\big(\Sphere^3_{\tau_m(t)}\big)$ such that $\psi\left(\operatorname{Im}\big(y^{\tau_m(t)}_{f_m(t)}\big)\right)$ is a Hopf torus.
    \end{enumerate}
\end{theorem}

By a continuity argument for the Willmore energy along the paths involved, we can also guarantee that the surfaces bifurcating from $\widetilde{\Sigma}_\tau$ are not isometric to those bifurcating from $\Sigma_\tau$, at least for a short time.  

\begin{proof}
The following arguments are similar to the ones given in the proof of \cref{theorem: Bifurcations of the Clifford Torus}. We highlight the differences. First, we study the eigenvalues $\tilde{\lambda}^+_{m,m}(\tau)$ of the second variation operator $\mathcal{L}_{\widetilde{\Sigma}_\tau}$.
\begin{claim}\label{claim: sinais dos autovalores do irmao bifurcado} 
It holds that
\begin{align*}
\tilde{\lambda}^+_{m,m}
    \begin{cases}
        <0, \quad \text{if $m \in \{0,1\}$}, \\
        >0, \quad \text{if $m \ge4$}.
    \end{cases}
\end{align*}
For $m \in \{2,3\}$, there exists a unique $\tilde{\tau}_m>\sqrt{2}$ such that $\tilde{\lambda}^+_{m,m}(\tilde{\tau}_m)=0$. Moreover,
\begin{equation*}
    \left.\frac{d}{d\tau}\right|_{\tau=\tilde{\tau}_m} \tilde{\lambda}^+_{m,m}(\tau) > 0.
\end{equation*}
Approximately, $\tilde{\tau}_2 \approx 1.767 $ and $\tilde{\tau}_3 \approx 1.439$.
\end{claim}
\begin{proof}
    It follows from \cref{proposition: spectrum of the irmao paralelo} that 
    \begin{align*}
        \tilde{\lambda}^+_{0,0}(\tau) &= 16 (\tau^4 - 3\tau^2 + 2), \\
        \tilde{\lambda}^+_{1,1}(\tau) &=  \frac{-16}{\tau^4}(\tau^2-1)(2\tau^4+1), \\
        \tilde{\lambda}^+_{m,m}(\tau) &=  \frac{16}{\tau^4}(\tau^2-1)P_m(\tau^2),     
    \end{align*}
    where
    \begin{equation*}
        P(u,m)=P_m(u)=(m^2-1)^2u^3-(3m^4-4m^2+2)u^2+3m^2(m^2-1)u -m^4.
    \end{equation*} 
    It is straightforward to verify $\tilde{\lambda}^+_{0,0}(\tau) < 0$ and $\tilde{\lambda}^+_{1,1}(\tau) < 0$, for all $\tau>\sqrt{2}$. Note that,
    \begin{align*}
        P_2(u) &= 9u^3 -38u^2 +36u -16, \\
        P_2(2) &= -24 < 0, \\
        P_2'(u) &= 27u^2 - 76u +36, \\
        P_3(u) &= 64u^3 - 218u^2 + 216u -81, \\
        P_3(2) &= -9 < 0, \\
        P_3'(u) &= 192u^2 - 436u + 216.
    \end{align*}
    Studying the sign of $P_m'$, $m \in \{2,3\}$, we see that $P_m$ has a unique root $\tilde{u}_m>2$ with $P_m'(\tilde{u}_m)>0$. Defining $\tilde{\tau}_m = \sqrt{\tilde{u}_m}$, by the chain rule, we obtain $\left.\frac{d}{d\tau}\right|_{\tau=\tilde{\tau}_m} \tilde{\lambda}^+_{m,m}(\tau) > 0$. We may find exact expressions for $\tilde{\tau}_m$ using the Cardano's formula. It only remains to verify that $\tilde{\lambda}^+_{m,m}(\tau)>0$, for all $\tau>\sqrt{2}$ and all $m\geq4$. It is enough to prove that $P(u,m)>0$ for all $u>2$, and all $m \geq 4$. Indeed,
    \begin{align*}
        P_m(2) &= m^2(m^2-10) > 0, \\
        P_m'(2) &= 3m^4 - 15m^2 +4 > 0, \\
    \end{align*}
    and
    \begin{align*}
         P_m''(u) &= 6(m^2-1)^2u-(6m^4-6m^2+4) \\ 
         &\geq 6m^4 - 18m^2+8 > 0.
    \end{align*}
    Hence, $P(u,m)>0$ under these constraints. 
    
\end{proof}
To better understand a tubular neighbourhood of $\widetilde{\Sigma}_\tau$ we observe the following. Let $\Sigma = \Sphere^1\big(\cos (\phi)\big) \times \Sphere^1\big(\sin(\phi)\big)$, $\phi \in \R$, and 
\begin{equation*}
    v \colon \T^2 \to \Sphere^3_1, \quad (z,w) \mapsto (\cos(\phi)z,\sin(\phi)w),
\end{equation*} be a parametrization of $\Sigma$. Let $f = \mu \in \R$ be a constant function on $\T^2$ and 
\begin{equation*}
    v_f \colon \T^2 \to \Sphere^3_1, \quad p \mapsto \exp^1_{v(p)}\big(f(p)n(p)\big), 
\end{equation*} 
be the parametrization of the graph of $f$ over $\Sigma$, where $n(z,w) = (-\sin(\phi)z,\cos(\phi)w)$ is a unit normal vector field along $v$. Then,
\begin{equation}\label{eq: graph of constant function over BB}
    \ima(v_f) = \Sphere^1\big(\cos(\phi+\mu)\big) \times \Sphere^1\big(\sin(\phi+\mu)\big).
\end{equation}
Denote $$
V(r,\tau)=\big\{q \in \Sphere^3 \mid d(q,\widetilde{\Sigma}_{\tau}) < r\big\},
$$
where $d$ is the Riemannian distance on $\Sphere^3_1$, and $$
B(r)= \Big\{ f \in C^{4,\alpha}_K(\T^2) \mid \sup\limits_{p \in \T^2} |f(p)| < r\Big\}.
$$ 
By compactness of $\T^2$ and \cref{eq: constantes do irmao bifurcado,eq: parametrização grafica do BB,eq: graph of constant function over BB}, for each $m \in \{2,3\}$, there exist $r_m > 0$, and a interval $I_m \subset(\sqrt{2},+\infty)$ containing $\tilde{\tau}_m$, such that $V(2r_m,\tilde{\tau}_m)$ is a tubular neighbourhood of $\widetilde{\Sigma}_{\tilde{\tau}_m}$ in $\Sphere^3_1$ and
\begin{enumerate} 
    \item[($P1$)] $\forall \tau \in I_m, \quad V(r_m,\tau)\subset V(2r_m,\tilde{\tau}_m)$ is a tubular neighbourhood of $\widetilde{\Sigma}_\tau$ in $\Sphere^3_1$. 
    \item[($P2$)] $\forall \tau \in I_m, \ \forall f\in B(r_m), \quad \ima(y^\tau_{f}) \subset V(r_m,\tau) \quad \text{and} \quad \swap (\widetilde{\Sigma}_{\tau}) \cap V(2r_m,\tilde{\tau}_m) = \varnothing$. \label{Propriedade da vizinhança tubular dos BBs}
\end{enumerate}

\medskip
Define
$$
\mathcal{\widetilde{P}} \colon I_m \times B(r_m) \to C^{0,\alpha}_K(\T^2), \quad (\tau,f) \mapsto \mathcal{E_W}(y^\tau_f).
$$
Thus, $y^\tau_f \colon \T^2 \to \berger$ is a Willmore surface if, and only if, $\mathcal{\widetilde{P}}(\tau,f) = 0$. Arguing as before and using $(P1)$, we see that $\mathcal{\widetilde{P}}$ is well--defined, and this map is smooth, since the family {$\{y^\tau\}_{\tau>\sqrt{2}}$} is smooth. From the fact that $\widetilde{\Sigma}_\tau$ is a Willmore surface of $\berger$, we have
\begin{equation*}
    \forall \tau \in I_m, \quad \mathcal{\widetilde{P}}(\tau,0) = 0.
\end{equation*}
It follows from \cref{proposition: second variation operator irmão paralelo,proposition: spectrum of the irmao paralelo,claim: sinais dos autovalores do irmao bifurcado}, that
$$
\widetilde{L}:=D_0(\mathcal{\widetilde{P}}(\tilde{\tau}_m, \cdot)) \colon C^{4,\alpha}_K(\T^2) \to C^{0,\alpha}_K(\T^2)
$$
is a Fredholm operator such that:
\begin{enumerate}[label=\textbullet]
    \item The Fredholm index of $\widetilde{L}$ is 0.
    \item $\operatorname{Ker}(\widetilde{L}) = \R C_{m,m}^+$ and $\frac{C^{0,\alpha}_K(\T^2)}{\operatorname{Im}(\widetilde{L})}$ is generated by the class of $C_{m,m}^+$.
    \item $\left.\frac{d}{d\tau}\right|_{\tau=\tilde{\tau}_m} D_0(\mathcal{\widetilde{P}}(\tau, \cdot))(C_{m,m}^+) \notin \ima(\widetilde{L})$.
\end{enumerate}
Thus, we may apply the Crandall--Rabinowitz \cref{Crandall-Rabinowitz} to bifurcate the path $\tau \mapsto (\tau,0)$, as roots of $\mathcal{\widetilde{P}}$, at $\tau=\tilde{\tau}_m$. More precisely, there exist $\delta_m>0$, an open neighbourhood
$$
\mathcal{U}_m \subset I_m \times C_K^{4,\alpha}(\T^2)
$$
of $(\tilde{\tau}_m,0)$, and a smooth path 
$$
\gamma_m \colon (-\delta_m,\delta_m) \to \mathcal{U}_m, \quad \gamma_m(t)=(\tau_m(t),f_m(t)),
$$ 
such that 
\begin{enumerate}
    \item[(i)] $\forall t \in (-\delta_m,\delta_m), \quad f_m(t) \in B(r_m)$.
    \item[(ii)] $\tau_m(0) = \tilde{\tau}_m$, $f_m(0) = 0 \in C^{4,\alpha}_K(\T^2)$, $f_m(t) \neq 0$ when $t \ne 0$, and $f_m'(0)=C_{m,m}^+$.
    \item[(iii)] $\{\mathcal{\widetilde{P}}(\tau,f) = 0\} \cap \mathcal{U}_m = \big(\{(\tau,0) : \tau \in I_m \} \cup \ima(\gamma_m)\big) \cap \mathcal{U}_m$.
\end{enumerate}
From $(ii)$ and $(iii)$, it follows that $\tilde{\tau}_m$ is a $K$--invariant bifurcation parameter for $\widetilde{\Sigma}_\tau$ as a Willmore surface such that properties $(a)$ of \cref{theorem: Bifurcations of the bifurcated brother} hold. By the \cref{claim: sinais dos autovalores do irmao bifurcado,lemma: Orthogonal basis of K-invariant functions}, if $\tau \ne \tilde{\tau}_m$ for each $m \in \{2,3\}$, then 
$$
D_0\big(\widetilde{\mathcal{P}}(\tau, \cdot )\big) \colon C^{4,\alpha}_K(\T^2) \to C^{0,\alpha}_K(\T^2)
$$
is a linear isomorphism. Therefore, by the implicit function theorem, the set all $K$--invariant bifurcation parameter of the Clifford Torus as a Willmore surface is exactly $\{\tilde{\tau}_2, \tilde{\tau}_3\}$.

Let us compute the variational vector field of $t \mapsto y^{\tau(t)}_{f_m(t)}$. Setting 
$$
\phi(\tau)=\frac{1}{2}\arccos\Big(\frac{\sqrt{\tau^2-2}}{\sqrt{\tau^2-1}}\Big),
$$ we have that $y^\tau(z,w) = \big(\cos\big(\phi(\tau)\big)z,\sin\big(\phi(\tau)\big) w\big)$. Thus,
$$
\forall \tau>0, \ \forall t \in \R, \ \forall p=(z,w) \in \T^2, \quad \exp^1_{y^\tau(p)}\big(t\widetilde{N}_\tau(p)\big) = \big(\cos(\phi(\tau)+t)z,\sin(\phi(\tau)+t)w\big),  
$$
from where we obtain
$$
\forall \tau,\tilde{\tau}>0, \ \forall p \in \T^2, \quad y^{\tau}(p) = \exp^1_{y^{\tilde{\tau}}(p)}\big((\phi(\tau)-\phi(\tilde{\tau}))\widetilde{N}_{\tilde{\tau}}(p)\big).
$$
Thus,
\begin{align*}
    \left.\frac{\partial}{\partial_t}\right|_{t=0} y^{\tau(t)}_{f_m(t)}(p) &= \left.\frac{\partial}{\partial_t}\right|_{t=0} \exp^1_{y^{\tau_m(t)}}\big(f_m(t)(p)\widetilde{N}_{\tau_m(t)}(p)\big) \\ 
    &= \left.\frac{\partial}{\partial_t}\right|_{t=0} \exp_{y^{\tau_m(0)}}^1\Big(\big(\phi(\tau_m(t))-\phi(\tau_m(0)) +f_m(t)(p)\big)\widetilde{N}_{\tau_m(0)}(p) \Big) \\
    &=\big(\phi'(\tilde{\tau}_m)\tau_m'(0) + C_{m,m}^+(p)\big)\widetilde{N}_{\tau_m}(p).
\end{align*}

\medskip
Now we verify property $(c)$: assume that there exist $m \in \{2,3\}$, $t \in (-\delta_m,\delta_m)$, an isometry $\psi \in \iso(\Sphere^3_{\tau_m(t)}) = U^+ \cup U^-$, and a Hopf torus $\hopftori \subset \Sphere^3_{\tau_m(t)} $ such that $\ima(y^{\tau_m(t)}_{f_m(t)}) = \psi(\hopftori)$. We will show that $t=0$. 

\medskip
Since the functions 
$$
s \mapsto W\big(\ima(y^{\tau_m(s)}_{f_m(s)})\big) \quad \text{and} \quad s \mapsto \willmore(\Sigma_{\tau_m(s)})
$$ 
are continuous, $\ima(y^{\tau_m(0)}_{f_m(0)}) = \widetilde \Sigma_{\tau_m}$, and $\Sigma_{\tau_m(0)} = \Sigma_{\tau_m}$, by \cref{remark: Willmore energy of Hopf tori family} we may assume that $\delta_m$ is sufficiently small to ensure that 
$$
\forall s \in (-\delta_m,\delta_m), \quad \willmore\big(\ima(y^{\tau(s)}_{f_m(s)})\big) < \willmore(\Sigma_{\tau_m(s)}).
$$ 
Since $f_m(t)$ is $K$--invariant and $\hopftori$ is a Hopf torus, we have that 
$$
H_1 = \{R_{\varphi,\varphi}\} \quad \text{and} \quad  H_2 = \psi^{-1}\{R_{\theta,-\theta}\}\psi
$$ 
are subgroups of $\iso_0(\hopftori)$. As in the Clifford case, we may use the Maximal Torus Theorem and \cref{claim: Rigidity of Symmetries of Hopf tori} to obtain an isometry $\zeta \in U^+$ such that $\zeta(\hopftori) = \widetilde{\Sigma}_{\tau(t)}$ and
$$
\{R_{\theta,-\theta}\} \subset \fix\Big((\psi \circ \zeta^{-1})\big(\widetilde{\Sigma}_{\tau_m(t)}\big)\Big).
$$
We may use \cref{lemma: Symetries of Clifford Torus} to perform the same computation done in the proof of \cref{claim: Rigidity of symmetries of clifford torus} and conclude that exist $\mu,\nu \in \R$ such that 
\begin{equation*}
    (\psi \circ \zeta^{-1}) \in  \{R_{\mu,\nu}, \ R_{\mu,\nu} \circ \swap, \ \conj \circ R_{\mu,\nu}, \ \conj \circ R_{\mu,\nu} \circ \swap \}.
\end{equation*}
Thus,
$$
(\psi \circ \zeta^{-1})\big(\widetilde{\Sigma}_{\tau_m(t)}\big) = \widetilde{\Sigma}_{\tau_m(t)} \quad \text{or} \quad (\psi \circ \zeta^{-1})\big(\widetilde{\Sigma}_{\tau_m(t)}\big)= \swap\big(\widetilde{\Sigma}_{\tau_m(t)}\big). 
$$
Since
$$
\ima\big(y^{\tau_m(t)}_{f_m(t)}\big)  = \psi(\hopftori) = (\psi \circ \zeta^{-1})\big(\widetilde{\Sigma}_{\tau(t)})\big),
$$
it follows from $(i)$, $(P1)$, and $(P2)$, that $\ima\big(y^{\tau_m(t)}_{f_m(t)}\big) = \widetilde{\Sigma}_{\tau(t)}$. 

\medskip
Since $\tilde{\tau}_m$ is a $K$--invariant bifurcation parameter of $\widetilde{\Sigma}_{\tau}$, we obtain that $f_m(t) = 0$ and $t = 0$ (see \cref{definition: invariant bifurcation}).

\end{proof}

As in the Clifford torus case, the argument based on the implicit function theorem, which shows that there are no other $K$--invariant bifurcation parameters for the bifurcated brother, yields the following stronger result.

\begin{corollary}[Local rigidity of the bifurcated brother as a Willmore surface]
Let $\tilde{\tau}$ be a nondegenerate $K$--invariant Berger parameter for the bifurcated brother as a Willmore surface, that is,
\begin{equation*}
    \tilde{\tau} \in (\sqrt{2},+\infty)\setminus\{\tilde{\tau}_1, \tilde{\tau}_2\}.
\end{equation*}
Then, there exist $\delta=\delta(\tilde{\tau})>0$ and an open neighborhood of the origin
\begin{equation*}
    U=U_{\tilde{\tau}}\subset C^{4,\alpha}_K(\T^2),
\end{equation*}
such that if $f\in U$ and $y^\tau_f:\T^2\to\berger$ is a Willmore surface for some $\tau\in(\tilde{\tau}-\delta,\tilde{\tau}+\delta)$, then $f=0$, and hence $y^\tau_f=y^\tau$ is the parametrization of the bifurcated brother.
\end{corollary}

 \section{Final remarks}\label{section: Final Remarks}

We compare the energies of $\Sigma_\tau$ and $\Sphere^2_\tau$ (see \cref{fig:Clifford vs Equador}). Using the coordinates given in \cref{ex: equador}, one can verify that 
\begin{equation*}
 |A_{\Sphere^2_\tau}|^2(a,\theta) = 2\tau^2(1-\tau^2)^2 \frac{(1-a^2)^2}{(a^2+(1-a^2)\tau^2)^2} \quad \text{and} \quad (\det g_{\Sphere^2_\tau})(a,\theta) = a^2 + (1-a^2)\tau^2.   
\end{equation*}
Using the software Mathematica to compute the integral of $|\traceless{A}_{\Sphere^2_\tau}|^2=|A_{\Sphere^2_\tau}|^2$, we obtain 
\begin{equation*}
    \willmore(\Sphere^2_\tau) = \frac{4 \pi}{\tau^2-1}\big(-2+\tau^2+\tau^4+\tau^2(\tau^2-4)h(\tau)\big),
\end{equation*}
where
\begin{equation*}
h(\tau)=
  \begin{cases}
      -\sqrt{1-\tau^2}\operatorname{ArcCsch\left(\frac{\tau}{\sqrt{1-\tau^2}}\right)}, \quad &\text{if} \quad 0<\tau<1, \\
      \sqrt{\tau^2-1}\operatorname{ArcCsc}\left(\frac{\tau}{\sqrt{\tau^2-1}}\right), \quad &\text{if} \quad \tau>1.
  \end{cases}
\end{equation*}
The Willmore energy of the Clifford torus was computed in \cref{remark: Willmore energy of Hopf tori family}.
\begin{figure}[ht]
    \centering
    \begin{subfigure}{0.48\textwidth}
        \centering
        \includegraphics[width=\linewidth]{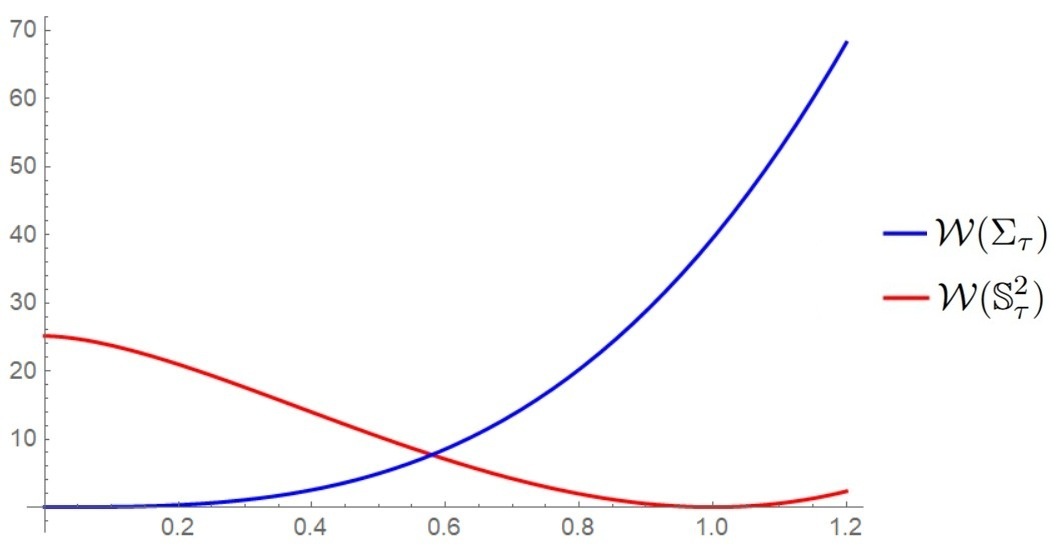}
        
        \label{fig:Clifford vs Equador 1}
    \end{subfigure}
    \hfill
    \begin{subfigure}{0.48\textwidth}
        \centering
        \includegraphics[width=\linewidth]{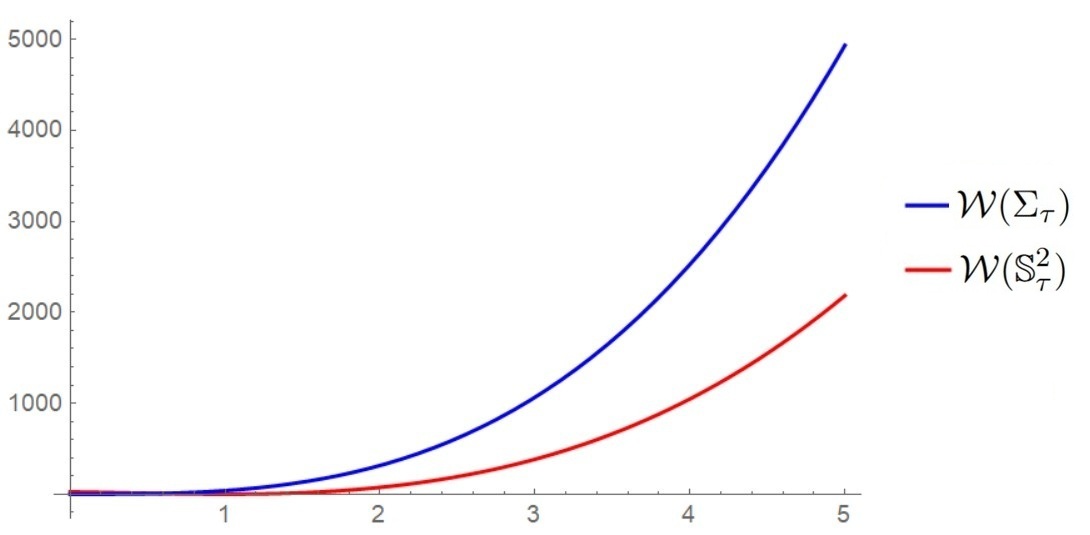}
        
        \label{fig:Clifford vs Equador 2}
    \end{subfigure}
    
    \caption{Willmore energy -- Clifford torus vs. equator. Observe the different scales.}
    \label{fig:Clifford vs Equador}
\end{figure}

\medskip
Concerning the discussion around CMC Willmore surfaces in Berger Spheres (see \cref{ex: equador,Prop: Non-existence of CMC Willmore surfaces}), one may ask whether there exists a non-minimal CMC surface among the bifurcating surfaces $x_{f_m(t)}$ in the range $\tau \in \left(1,\sqrt{2}\right]$. We claim that these surfaces have points with negative and positive mean curvature when they are not one of the two explicit examples (see \cref{remark: trivial bifurcation at 1,remark: trivial bifurcation at sqrt 2}) and $t$ is sufficiently small. Indeed, denoting by $H(f,\tau)$ the mean curvature of $x_f \colon \T^2 \to \berger$, since $\Sigma_\tau$ is minimal for every $\tau>0$ and satisfies
\begin{align*}
    |A_\tau|^2 &= 2\tau^2, \\
    \ric_\tau(N,N) &= 2(2-\tau^2), \\
    \Delta_\tau &= \Delta_1 + (1-\tau^2)L_\xi^2,
\end{align*}
we have
\begin{equation*}
     H\big(f_m(0),\tau_m(0)\big)=0,
\end{equation*}
and
\begin{align*}
    \left.\frac{d}{dt}\right|_{t=0} H\big(f_m(t),\tau_m(t)\big) &= D_{f_m(0)}(H(\cdot,\tau_m))(f_m'(0)) \\ 
    &= (\Delta_{\tau_m} + 4)C_{m,m}^+ \\
    &=(\Delta_1 + (1-\tau_m^2)L_\xi^2 + 4)C_{m,m}^+ \\
    &= \frac{4}{\tau_m^2}(\tau_m^2-m^2)C_{m,m}^+.
\end{align*}
By \cref{proposition: spectrum of the Clifford torus}, we have $\tau_m < m$ for $m \geq 2$. Since there exist points $p,q \in \T^2$ such that $C_{m,m}^+(p) > 0$ and $C_{m,m}^+(q) < 0$, the claim follows.

\medskip
The same argument does not apply to the bifurcating surfaces $y^{\tau_m(t)}_{f_m(t)}$, because in this case the time derivative of the mean curvature depends on $\tau_m'(0)$, for which we have no information.

\medskip
We provide additional information to \cref{theorem: Stability of the Clifford torus}. It follows from the stability result about elasticas in the round $\Sphere^2$ \cite[Th. 3.1.]{Langer&Singer1987} that, when we restrict the Willmore functional to the Hopf tori family, the Clifford torus $\Sigma_\tau$ is stable in the range $\tau \in (0,\sqrt{2}]$, and its index is at least one when $\tau > \sqrt{2}$. The Hopf tori family is the $H$--invariant tori family, where $H = \{R_{\theta,\theta} \mid \theta \in \R \}$. By \cref{proposition: spectrum of the Clifford torus}, the eigenvalues corresponding to $H$--invariant eigenfunctions of $\mathcal{L}_{\Sigma_\tau}$ are of the form
\begin{equation*}
\lambda_{m,m}^-(\tau) = 16(m^2(m^2-2)+1) + 8(m^2-1)\tau^2,
\end{equation*}
for $m \geq 0$. Thus, within this family, the index of $\Sigma_\tau$ is exactly one in the range $\tau > \sqrt{2}$. 

\medskip
Finally, we also expect the existence of $G$--invariant bifurcations of the Clifford torus and the bifurcated brother associated with different symmetry subgroups $G$. For example, when $G = \operatorname{span}\{R_{0,\theta}, \conj \mid \theta \in \mathbb{R}\}$, the graphs of the eigenvalues $\lambda_{m,0}^+$ and $\tilde{\lambda}_{m,0}^+$, corresponding to $G$--invariant eigenfunctions of $\mathcal{L}_{\Sigma_\tau}$ and $\mathcal{L}_{\widetilde{\Sigma}_\tau}$, indicate that some of these eigenvalues change sign. Hence, the same approach should yield the existence of $G$--invariant bifurcation parameters of $\Sigma_\tau$ and $\widetilde{\Sigma}_\tau$ as Willmore surfaces.

 \appendix
   \section{Auxiliary results}\label{app:aux}

Recall the classical Crandall--Rabinowitz bifurcation theorem \cite[Th. 1.7.]{CrandallRabinowitz1971},
also known in the literature as \emph{bifurcation from a simple eigenvalue}.

\begin{theorem}[Crandall-Rabinowitz]\label{Crandall-Rabinowitz}
Let $X$ and $Y$ be real Banach spaces, $U \subset X$ open, $I \subset \R$ an interval, $x_0 \in U$ and $\mathcal{P} \colon I \times U \to Y$ a map of class $C^l$ with $l \in \{2,...,\infty\}$ such that 
$$
\forall \mu \in I, \quad \mathcal{P}(\mu,x_0) = 0.
$$ 
Suppose that there exists $\mu_0 \in I$ with the following properties:
\begin{enumerate}
    \item[\textnormal{(a)}] $L:=D_{x_0}(\mathcal{P}(\mu_0,.))$ is a Fredholm operator with Fredholm index $0$.
    \item[\textnormal{(b)}] There exist $\xi_0 \in X \backslash \{0\}$ with $\operatorname{Ker}(L) = \R \xi_0$.
    \item[\textnormal{(c)}] $\left.\frac{d}{d\mu} \right|_{\mu=\mu_0} D_{x_0} (\mathcal{P}(\mu, .))(\xi_0) \notin L(X)$. 
\end{enumerate}
Then, there exist a neighbourhood $V \subset I \times X$ of $(\mu_0,x_0)$ and a path of class $C^{l-1}$
$$
\gamma \colon (-\delta,\delta) \to V, \quad \gamma(t)=(\mu(t),x(t)),
$$
such that:
\begin{itemize}
    \item[\textnormal{(1)}] $\mu(0)=\mu_0$, $x(0)=x_0$, $x(t) \neq x_0$, when $t \neq 0$, and $x'(0) = \xi_0$,
    \item[\textnormal{(2)}] $\{\mathcal{P}(\mu,x) = 0\} \cap V = \big(\{(\mu,x_0) \}_{\mu \in I} \cup \ima(\gamma)\big) \cap V$.
\end{itemize}

\end{theorem}
   \section{Proofs of auxiliary lemmas}\label{app: lemma's proofs}

\begin{proof}[Proof of \cref{lemma: Symetries of Clifford Torus}]
We may regard $\Sigma_\tau = \Sigma$ in the case $r_1 = r_2$. Let $(z,w) \in \Sigma$ and $\psi \in  \fix(\Sigma) \subset U^+ \cup U^-$.
Assume first that 
\begin{equation}\label{eq: relations of entries in unitary matrix}
    \psi = 
    \begin{pmatrix}
        \alpha & \beta \\
        \gamma & \delta
    \end{pmatrix} \in U^+, \quad
     \begin{cases}
        |\alpha|^2 + |\beta|^2 = |\gamma|^2 + |\delta|^2 = 1, \\
        \alpha\overline{\gamma} + \beta \overline{\delta} = 0.
    \end{cases}
\end{equation}
Since $\psi(z,w) \in \Sigma = \Sphere^1(r_1) \times \Sphere^1(r_2)$, we obtain
\begin{align}
  \begin{cases}\label{eq: image of (z w)}
    r_1^2 &= |\alpha z + \beta w|^2 = |\alpha|^2r_1^2 + |\beta|^2r_2^2 + 2\operatorname{Re}(\alpha\overline{\beta}z\overline{w}), \\
    r_2^2 &= |\gamma z + \delta w|^2 = |\gamma|^2 r_1^2 + |\delta|^2 r_2^2 + 2\operatorname{Re}(\gamma\overline{\delta}z\overline{w}).
  \end{cases}
\end{align}
Subtracting $|\gamma z + \delta w|^2$ from $|\alpha z + \beta w|^2$ and using \cref{eq: image of (z w),eq: relations of entries in unitary matrix}, we obtain
\begin{align}\label{eq: equacao das simetrias dos toros paralelos}
     (r_1^2-r_2^2)(1 + |\beta|^2 - |\delta|^2) &= 2\operatorname{Re}\big((\alpha\overline{\beta} - \gamma \overline{\delta})z\overline{w} \big).
\end{align}
\emph{Case: $r_1 \neq r_2$}. Assume that $\alpha\overline{\beta} - \gamma \overline{\delta} \neq 0$. Choose $w = r_2 \in \Sphere^1(r_2)$ and $z = ir_1\frac{\overline{\alpha} \beta - \overline{\gamma} \delta}{|\overline{\alpha}\beta - \overline{\gamma} \delta|} \in \Sphere^1(r_1)$. From \cref{eq: equacao das simetrias dos toros paralelos}, since $r_1^2 - r_2^2 \neq 0$ and $\operatorname{Re}\big((\alpha\overline{\beta} - \gamma \overline{\delta})z\overline{w} \big) = 0$, we get $|\delta|^2 - |\beta|^2 = 1$. It follows from \cref{eq: relations of entries in unitary matrix} that $\beta = 0$ and $\gamma = 0$, which contradicts our assumption. Thus, $\alpha\overline{\beta} - \gamma \overline{\delta} = 0$.

\medskip
\noindent
\emph{Case: $r_1 = r_2$.} The \cref{eq: image of (z w)} becomes
$\operatorname{Re}\big((\alpha\overline{\beta} - \gamma \overline{\delta})z\overline{w} \big) = 0$.
Choosing $w = r_2 \in \Sphere^1(r_2)$ and letting $z$ range over $\Sphere^1(r_1)$, we see that $\alpha\overline{\beta} - \gamma \overline{\delta} = 0$.

\medskip
\noindent
Either case, we obtain $\alpha\overline{\beta} - \gamma \overline{\delta} = 0$. It follows from \cref{eq: relations of entries in unitary matrix} that if $\delta= 0$, then $\alpha = 0$ and $\psi = R_{\mu,\nu} \circ \swap$ for some $\mu,\nu \in \R$. This implies that $\swap \in \fix(\Sigma)$, which can occur only in the case $r_1=r_2$. If $\delta \neq 0$, by \cref{eq: relations of entries in unitary matrix} we see that $\beta=0$ and $\gamma = 0$. Thus, $\psi = R_{\mu,\nu}$ for some $\mu,\nu \in \R$.

\medskip
\noindent
Now assume that $\psi = \tilde{\psi} \circ \conj \in U^-$ with $\tilde{\psi} \in U^+$. Since $\conj \in \fix(\Sigma)$, we also have $\tilde{\psi} \in \fix(\Sigma)$. By what we have already proved, if $r_1 \neq r_2$, there exist $\mu,\nu \in \R$ such that $\psi = R_{\mu,\nu} \circ \conj$. If $r_1 = r_2$, there exist $\mu,\nu \in \R$ such that $\psi = R_{\mu,\nu} \circ \conj$ or $\psi = R_{\mu,\nu} \circ \swap \circ \conj$. The remaining assertions follow routinely.

\end{proof}

\begin{proof}[Proof of \cref{lemma: Orthogonal basis of K-invariant functions}]

    By \cref{eq: K-invariant functions on T2} a function $f \colon \T^2 \to \R$ is $K$--invariant exactly when 
    \begin{equation*}
    \forall \theta,\varphi,\eta \in \R \colon \quad f(e^{i(\theta+\eta)},e^{i(\varphi-\eta)}) = f(e^{i\theta},e^{i\varphi}) = f(e^{-i\theta},e^{-i\varphi}).
    \end{equation*}
    Thus, $C_{m,m}^+$ is $K$--invariant. We recall that there is an orthogonal Hilbert basis 
    \begin{equation*}
    \mathcal{B} \subset \{C_{m,n}^\pm,S_{m,n}^\pm\}_{m,n \in \N_0}
    \end{equation*}
    of both
    \begin{equation*}
    \big(L^2(\T^2),\ip{\cdot}{\cdot}_{L^2}\big) \quad \text{and} \quad \big(H^4(\T^2),\ip{\cdot}{\cdot}_{H^4}\big). 
    \end{equation*}
    By density, it is sufficient to show that $S_{m,n}^\pm$, $C_{m,n}^{-}$, and $C_{p,q}^+$ are $L^2$--orthogonal and $H^4$--orthogonal to a function $u \in C_K^{0,\alpha}(\T^2)$ and a function $v \in C_K^{4,\alpha}(\T^2)$, respectively, for all $(m,n),(p,q) \in (\N_0)^2 \setminus\{(0,0)\}$ with $p \ne q$. The $L^2$--orthogonality follows from the trigonometric properties of $\sin$ and $\cos$, the $K$--invariance of $u$, and the Change of Variables Theorem. For example, consider the coordinate system $(\theta,\varphi) \in (0,2\pi)^2 \mapsto (e^{i\theta},e^{i\varphi}) \in \T^2$. Then,

    \begin{align*}
        \ip{S_{m,n}^-}{u}_{L^2} &= \int_{\T^2} S_{m,n}^- u \, d\T^2 \\
                                &= \int_0^{2\pi} \int_0^{2\pi} S_{m,n}^-(\theta,\varphi)u(\theta,\varphi) d\theta d\varphi \\
                                &= \frac{1}{2} \int_0^{2\pi} \int_{-\pi}^\pi S_{m,n}^-(a+b,a-b) \, u(a+b,a-b)\, db\, da \\
                                &= \frac{1}{2} \int_0^{2\pi} u(a,a) \int_{-\pi}^\pi S_{m,n}^-(a+b,a-b) \, db\, da.
    \end{align*}

    On the other hand,

    \begin{align*}
        \int_{-\pi}^\pi S_{m,n}^-(a+b,a-b) \, db = \frac{1}{m+n} \int_{-(m+n)\pi}^{(m+n)\pi} \sin(t) \, dt = 0.
    \end{align*}

    Thus, $\ip{S_{m,n}^-}{u}_{L^2} = 0$.

    \medskip
    Now we verify the $H^4$--orthogonality. Notice first that 
    $$
    \Delta_{can}v \in C^{2,\alpha}_K(\T^2) \quad \text{and} \quad \Delta_{can}^2v \in C^{0,\alpha}_K(\T^2).
    $$ Indeed, since $\varphi_k$ is an isometry, for every $k \in K$, it holds that
    \begin{align*}
        \Delta_{can}v = \Delta_{can}(v \circ \varphi_k) = (\Delta_{can}v) \circ \varphi_k.
    \end{align*}
    Similarly, $\Delta^2_{can}v$ is $K$--invariant. The $H^4$--orthogonality follows from the $L^2$-orthogonality verified above and from the fact that $S_{m,n}^\pm$, $C_{m,n}^{-}$, and $C_{p,q}^+$ are eigenfunctions of $\Delta_{can}$. For example, 
    \begin{align*}
        \ip{S_{m,n}^+}{v}_{H^4} &= \int_{\T^2} (I-\Delta_{can})^2S_{m,n}^+ \, (I-\Delta_{can})^2v \, d\T^2 \\
                        &= \int_{\T^2}\big(S_{m,n}^+ + 2(m^2+n^2)S_{m,n}^+ + (m^2+n^2)^2 S_{m,n}^+\big)\big(v - 2 \Delta_{can}v + \Delta_{can}^2 v \big) \, d\T^2  \\
                        &= (1+m^2+n^2)^2 \int_{\T^2} S_{m,n}^+ \big(v - 2 \Delta_{can}v + \Delta_{can}^2v \big) d\T^2  \\
                        &= 0.
    \end{align*}    
\end{proof}
 \bibliographystyle{alpha}  
 \bibliography{references}

@article{Fabio&Alma2022,
  author    = {Albujer, Alma L. and dos Santos, F{\'a}bio R.},
  title     = {Willmore surfaces and Hopf tori in homogeneous 3-manifolds},
  journal   = {Annals of Global Analysis and Geometry},
  year      = {2022},
  volume    = {62},
  number    = {3},
  pages     = {403--420},
  doi       = {10.1007/s10455-022-09844-2},
  url       = {https://doi.org/10.1007/s10455-022-09844-2},
}

@article {Marques&Neves,
    AUTHOR = {Marques, Fernando C. and Neves, Andr\'e},
     TITLE = {Min-max theory and the {W}illmore conjecture},
   JOURNAL = {Ann. of Math. (2)},
  FJOURNAL = {Annals of Mathematics. Second Series},
    VOLUME = {179},
      YEAR = {2014},
    NUMBER = {2},
     PAGES = {683--782},
      ISSN = {0003-486X,1939-8980},
   MRCLASS = {53C42 (49Q20)},
  MRNUMBER = {3152944},
MRREVIEWER = {Andrea\ Mondino},
       DOI = {10.4007/annals.2014.179.2.6},
       URL = {https://doi.org/10.4007/annals.2014.179.2.6},
}

@article {Willmore1965,
    AUTHOR = {Willmore, T. J.},
     TITLE = {Note on embedded surfaces},
   JOURNAL = {An. Sti. Univ. ``Al. I. Cuza'' Ia si},
    VOLUME = {11B},
      YEAR = {1965},
     PAGES = {493--496},
      ISSN = {0041-9109},
   MRCLASS = {53.01 (53.74)},
  MRNUMBER = {202066},
MRREVIEWER = {H.\ B.\ Griffiths},
}

@article {CrandallRabinowitz1971,
    AUTHOR = {Crandall, Michael G. and Rabinowitz, Paul H.},
     TITLE = {Bifurcation from simple eigenvalues},
   JOURNAL = {J. Functional Analysis},
  FJOURNAL = {Journal of Functional Analysis},
    VOLUME = {8},
      YEAR = {1971},
     PAGES = {321--340},
      ISSN = {0022-1236},
   MRCLASS = {47.80 (35.00)},
  MRNUMBER = {288640},
MRREVIEWER = {B.\ V.\ Loginov},
       DOI = {10.1016/0022-1236(71)90015-2},
       URL = {https://doi.org/10.1016/0022-1236(71)90015-2},
}

@article {Manuel2014,
    AUTHOR = {Barros, Manuel and Ferr\'andez, Angel and Garay, \'Oscar J.},
     TITLE = {Equivariant {W}illmore surfaces in conformal homogeneous three
              spaces},
   JOURNAL = {J. Math. Anal. Appl.},
  FJOURNAL = {Journal of Mathematical Analysis and Applications},
    VOLUME = {409},
      YEAR = {2014},
    NUMBER = {1},
     PAGES = {459--477},
      ISSN = {0022-247X,1096-0813},
   MRCLASS = {53C42 (53C30 58E30)},
  MRNUMBER = {3095054},
MRREVIEWER = {Mohamad\ Chaichi},
       DOI = {10.1016/j.jmaa.2013.07.031},
       URL = {https://doi.org/10.1016/j.jmaa.2013.07.031},
}

@article {Manuel1997,
    AUTHOR = {Barros, Manuel},
     TITLE = {Willmore tori in non-standard {$3$}-spheres},
   JOURNAL = {Math. Proc. Cambridge Philos. Soc.},
  FJOURNAL = {Mathematical Proceedings of the Cambridge Philosophical
              Society},
    VOLUME = {121},
      YEAR = {1997},
    NUMBER = {2},
     PAGES = {321--324},
      ISSN = {0305-0041,1469-8064},
   MRCLASS = {53C42},
  MRNUMBER = {1426526},
MRREVIEWER = {Udo\ Hertrich-Jeromin},
       DOI = {10.1017/S0305004196001466},
       URL = {https://doi.org/10.1017/S0305004196001466},
}

@article {LangerSinger1984,
    AUTHOR = {Langer, Joel and Singer, David A.},
     TITLE = {The total squared curvature of closed curves},
   JOURNAL = {J. Differential Geom.},
  FJOURNAL = {Journal of Differential Geometry},
    VOLUME = {20},
      YEAR = {1984},
    NUMBER = {1},
     PAGES = {1--22},
      ISSN = {0022-040X,1945-743X},
   MRCLASS = {58E10 (53C22)},
  MRNUMBER = {772124},
MRREVIEWER = {Gudlaugur\ Thorbergsson},
       URL = {http://projecteuclid.org/euclid.jdg/1214438990},
}

@book {BesseEinsteinManifolds,
    AUTHOR = {Besse, Arthur L.},
     TITLE = {Einstein manifolds},
    SERIES = {Ergebnisse der Mathematik und ihrer Grenzgebiete (3) [Results
              in Mathematics and Related Areas (3)]},
    VOLUME = {10},
 PUBLISHER = {Springer-Verlag, Berlin},
      YEAR = {1987},
     PAGES = {xii+510},
      ISBN = {3-540-15279-2},
   MRCLASS = {53C25 (53-02 53C21 53C30 53C55 58D17 58E11)},
  MRNUMBER = {867684},
MRREVIEWER = {S.\ M.\ Salamon},
       DOI = {10.1007/978-3-540-74311-8},
       URL = {https://doi.org/10.1007/978-3-540-74311-8},
}

@book {MorreyMultipleIntegrals,
    AUTHOR = {Morrey, Jr., Charles B.},
     TITLE = {Multiple integrals in the calculus of variations},
    SERIES = {Die Grundlehren der mathematischen Wissenschaften},
    VOLUME = {Band 130},
 PUBLISHER = {Springer-Verlag New York, Inc., New York},
      YEAR = {1966},
     PAGES = {ix+506},
   MRCLASS = {49.00 (00.00)},
  MRNUMBER = {202511},
MRREVIEWER = {M.\ Schechter},
}

@article {Langer&Singer1987,
    AUTHOR = {Langer, Joel and Singer, David A.},
     TITLE = {Curve-straightening in {R}iemannian manifolds},
   JOURNAL = {Ann. Global Anal. Geom.},
  FJOURNAL = {Annals of Global Analysis and Geometry},
    VOLUME = {5},
      YEAR = {1987},
    NUMBER = {2},
     PAGES = {133--150},
      ISSN = {0232-704X},
   MRCLASS = {58E10 (53C22)},
  MRNUMBER = {944778},
MRREVIEWER = {Dennis\ M.\ DeTurck},
       DOI = {10.1007/BF00127856},
       URL = {https://doi.org/10.1007/BF00127856},
}

@book {Hall2015,
    AUTHOR = {Hall, Brian C.},
     TITLE = {Lie groups, {L}ie algebras, and representations},
    SERIES = {Graduate Texts in Mathematics},
    VOLUME = {222},
   EDITION = {Second},
      NOTE = {An elementary introduction},
 PUBLISHER = {Springer, Cham},
      YEAR = {2015},
     PAGES = {xiv+449},
      ISBN = {978-3-319-13466-6},
   MRCLASS = {22-01 (17-01)},
  MRNUMBER = {3331229},
       DOI = {10.1007/978-3-319-13467-3},
       URL = {https://doi.org/10.1007/978-3-319-13467-3},
}

@article {Brian1987,
    AUTHOR = {White, Brian},
     TITLE = {The space of {$m$}-dimensional surfaces that are stationary
              for a parametric elliptic functional},
   JOURNAL = {Indiana Univ. Math. J.},
  FJOURNAL = {Indiana University Mathematics Journal},
    VOLUME = {36},
      YEAR = {1987},
    NUMBER = {3},
     PAGES = {567--602},
      ISSN = {0022-2518,1943-5258},
   MRCLASS = {58E15 (49F15)},
  MRNUMBER = {905611},
MRREVIEWER = {Helmut\ Kaul},
       DOI = {10.1512/iumj.1987.36.36031},
       URL = {https://doi.org/10.1512/iumj.1987.36.36031},
}

@article {Carlotto&Mondino2014,
    AUTHOR = {Carlotto, Alessandro and Mondino, Andrea},
     TITLE = {Existence of generalized totally umbilic 2-spheres in
              perturbed 3-spheres},
   JOURNAL = {Int. Math. Res. Not. IMRN},
  FJOURNAL = {International Mathematics Research Notices. IMRN},
  VOLUME = {2014},
      YEAR = {2014},
    NUMBER = {21},
     PAGES = {6020--6052},
      ISSN = {1073-7928,1687-0247},
   MRCLASS = {53C42},
  MRNUMBER = {3273070},
MRREVIEWER = {Maria\ Helena\ Noronha},
       DOI = {10.1093/imrn/rnt155},
       URL = {https://doi.org/10.1093/imrn/rnt155},
}

@article {Marque&NevesSurvey2014,
    AUTHOR = {Marques, Fernando C. and Neves, Andr\'e},
     TITLE = {The {W}illmore conjecture},
   JOURNAL = {Jahresber. Dtsch. Math.-Ver.},
  FJOURNAL = {Jahresbericht der Deutschen Mathematiker-Vereinigung},
    VOLUME = {116},
      YEAR = {2014},
    NUMBER = {4},
     PAGES = {201--222},
      ISSN = {0012-0456,1869-7135},
   MRCLASS = {53A10 (53C42)},
  MRNUMBER = {3280571},
MRREVIEWER = {Atsushi\ Fujioka},
       DOI = {10.1365/s13291-014-0104-8},
       URL = {https://doi.org/10.1365/s13291-014-0104-8},
}

@incollection {Hu&Li2004,
    AUTHOR = {Hu, Z. and Li, H.},
     TITLE = {Willmore submanifolds in a {R}iemannian manifold},
 BOOKTITLE = {Contemporary geometry and related topics},
     PAGES = {251--275},
 PUBLISHER = {World Sci. Publ., River Edge, NJ},
      YEAR = {2004},
      ISBN = {981-238-432-4},
   MRCLASS = {53C42},
  MRNUMBER = {2070875},
MRREVIEWER = {Shu-Cheng\ Chang},
       DOI = {10.1142/9789812703088\_0013},
       URL = {https://doi.org/10.1142/9789812703088_0013},
}

@article {Mondino&Riviere2013,
    AUTHOR = {Mondino, Andrea and Rivi\`ere, Tristan},
     TITLE = {Willmore spheres in compact {R}iemannian manifolds},
   JOURNAL = {Adv. Math.},
  FJOURNAL = {Advances in Mathematics},
    VOLUME = {232},
      YEAR = {2013},
     PAGES = {608--676},
      ISSN = {0001-8708,1090-2082},
   MRCLASS = {53A30 (35R01 49J53 58E15)},
  MRNUMBER = {2989995},
MRREVIEWER = {Basil\ J.\ Papantoniou},
       DOI = {10.1016/j.aim.2012.09.014},
       URL = {https://doi.org/10.1016/j.aim.2012.09.014},
}

@misc{wang2025willmoresurfaces4dimensionalconformal,
      title={Willmore surfaces in 4-dimensional conformal manifolds}, 
      author={Changping Wang and Zhenxiao Xie},
      year={2025},
      eprint={2306.00846},
      archivePrefix={arXiv},
      primaryClass={math.DG},
      url={https://arxiv.org/abs/2306.00846}, 
}

@article{Souam&Tobiana2009,
 author = {Souam, Rabah and Toubiana, Eric},
 title = {Totally umbilic surfaces in homogeneous 3-manifolds},
 fjournal = {Commentarii Mathematici Helvetici},
 journal = {Comment. Math. Helv.},
 issn = {0010-2571},
 volume = {84},
 number = {3},
 pages = {673--704},
 year = {2009},
 doi = {10.4171/CMH/177},
}

@article{Weiner1978,
 author = {Weiner, Joel L.},
 title = {On a problem of {Chen}, {Willmore} et al},
 fjournal = {Indiana University Mathematics Journal},
 journal = {Indiana Univ. Math. J.},
 issn = {0022-2518},
 volume = {27},
 pages = {19--35},
 year = {1978},
 doi = {10.1512/iumj.1978.27.27003},
}

@article{Bryant1984,
 author = {Bryant, Robert L.},
 title = {A duality theorem for {Willmore} surfaces},
 fjournal = {Journal of Differential Geometry},
 journal = {J. Differ. Geom.},
 issn = {0022-040X},
 volume = {20},
 pages = {23--53},
 year = {1984},
 doi = {10.4310/jdg/1214438991},
}

@article{Pinkall1985,
 author = {Pinkall, U.},
 title = {Hopf tori in {{\(S^ 3\)}}},
 fjournal = {Inventiones Mathematicae},
 journal = {Invent. Math.},
 issn = {0020-9910},
 volume = {81},
 pages = {379--386},
 year = {1985},
 doi = {10.1007/BF01389060},

}

@article{Kusner1987,
 author = {Kusner, Rob},
 title = {Conformal geometry and complete minimal surfaces},
 fjournal = {Bulletin of the American Mathematical Society. New Series},
 journal = {Bull. Am. Math. Soc., New Ser.},
 issn = {0273-0979},
 volume = {17},
 pages = {291--295},
 year = {1987},
 doi = {10.1090/S0273-0979-1987-15564-9},
}

@article{Kusner&Cya2024,
 author = {Hirsch, Jonas and Kusner, Rob and M{\"a}der-Baumdicker, Elena},
 title = {Geometry of complete minimal surfaces at infinity and the {Willmore} index of their inversions},
 fjournal = {Calculus of Variations and Partial Differential Equations},
 journal = {Calc. Var. Partial Differ. Equ.},
 issn = {0944-2669},
 volume = {63},
 number = {8},
 pages = {30},
 note = {Id/No 190},
 year = {2024},
 doi = {10.1007/s00526-024-02792-8},
}

@article{Riviere&Bernard2014,
 author = {Bernard, Yann and Rivi{\`e}re, Tristan},
 title = {Energy quantization for {Willmore} surfaces and applications},
 fjournal = {Annals of Mathematics. Second Series},
 journal = {Ann. Math. (2)},
 issn = {0003-486X},
 volume = {180},
 number = {1},
 pages = {87--136},
 year = {2014},
 doi = {10.4007/annals.2014.180.1.2},
}

@misc{Bettiol&Piccione2025,
      title={Bifurcations of Clifford tori in ellipsoids}, 
      author={Renato G. Bettiol and Paolo Piccione},
      year={2025},
      eprint={2309.13758},
      archivePrefix={arXiv},
      primaryClass={math.DG},
      url={https://arxiv.org/abs/2309.13758}, 
}

@article{Bettiol&Piccione2025-2,
   title={Nonplanar minimal spheres in ellipsoids of revolution},
   ISSN={0391-173X},
   url={http://dx.doi.org/10.2422/2036-2145.202309_018},
   DOI={10.2422/2036-2145.202309_018},
   journal={ANNALI SCUOLA NORMALE SUPERIORE - CLASSE DI SCIENZE},
   publisher={Scuola Normale Superiore - Edizioni della Normale},
   author={Renato G. Bettiol and Paolo Piccione},
   year={2024},
   month=May, pages={32} 
}

@article{Bettiol&Piccione2020,
 author = {Renato G. Bettiol and Paolo Piccione},
 title = {Instability and bifurcation},
 fjournal = {Notices of the American Mathematical Society},
 journal = {Notices Am. Math. Soc.},
 issn = {0002-9920},
 volume = {67},
 number = {11},
 pages = {1679--1691},
 year = {2020},
 doi = {10.1090/noti2185},
}

@article{Bettiol&Piccione2016,
 author = {Renato G. Bettiol and Paolo Piccione},
 title = {Delaunay-type hypersurfaces in cohomogeneity one manifolds},
 fjournal = {IMRN. International Mathematics Research Notices},
 journal = {Int. Math. Res. Not.},
 issn = {1073-7928},
 volume = {2016},
 number = {10},
 pages = {3124--3162},
 year = {2016},
 doi = {10.1093/imrn/rnv231},
}

@article{Bettiol&Piccione2022,
 author = {Bettiol, Renato G. and Piccione, Paolo},
 title = {Global bifurcation for a class of nonlinear {ODEs}},
 fjournal = {S{\~a}o Paulo Journal of Mathematical Sciences},
 journal = {S{\~a}o Paulo J. Math. Sci.},
 issn = {1982-6907},
 volume = {16},
 number = {1},
 pages = {486--507},
 year = {2022},
 doi = {10.1007/s40863-022-00290-3},
}

@article{Lima&Lira2014,
 author = {De Lima, L. L. and De Lira, J. H. and Piccione, P.},
 title = {Bifurcation of {Clifford} tori in {Berger} {{\(3\)}}-spheres},
 fjournal = {The Quarterly Journal of Mathematics},
 journal = {Q. J. Math.},
 issn = {0033-5606},
 volume = {65},
 number = {4},
 pages = {1345--1362},
 year = {2014},
 doi = {10.1093/qmath/hau003},
}

@article{Luis&Piccione2013,
 author = {Al{\'{\i}}as, Luis J. and Piccione, Paolo},
 title = {Bifurcation of constant mean curvature tori in {Euclidean} spheres},
 fjournal = {The Journal of Geometric Analysis},
 journal = {J. Geom. Anal.},
 issn = {1050-6926},
 volume = {23},
 number = {2},
 pages = {677--708},
 year = {2013},
 doi = {10.1007/s12220-011-9260-6},
}

@article{Daniel2007,
 author = {Daniel, Benoît},
 title = {Isometric immersions into 3-dimensional homogeneous manifolds},
 fjournal = {Commentarii Mathematici Helvetici},
 journal = {Comment. Math. Helv.},
 issn = {0010-2571},
 volume = {82},
 number = {1},
 pages = {87--131},
 year = {2007},
 doi = {10.4171/CMH/86},
}

@article{Torralbo&Urbano2010,
 author = {Torralbo, Francisco},
 title = {Rotationally invariant constant mean curvature surfaces in homogeneous 3-manifolds},
 fjournal = {Differential Geometry and its Applications},
 journal = {Differ. Geom. Appl.},
 issn = {0926-2245},
 volume = {28},
 number = {5},
 pages = {593--607},
 year = {2010},
 language = {English},
 doi = {10.1016/j.difgeo.2010.04.007},
}

@article{Torralbo&Urbano2012,
 author = {Torralbo, Francisco and Urbano, Francisco},
 title = {Compact stable constant mean curvature surfaces in homogeneous 3-manifolds},
 fjournal = {Indiana University Mathematics Journal},
 journal = {Indiana Univ. Math. J.},
 issn = {0022-2518},
 volume = {61},
 number = {3},
 pages = {1129--1156},
 year = {2012},
 doi = {10.1512/iumj.2012.61.4667},
}

@article{Torralbo2012,
 author = {Torralbo, Francisco},
 title = {Compact minimal surfaces in the {Berger} spheres},
 fjournal = {Annals of Global Analysis and Geometry},
 journal = {Ann. Global Anal. Geom.},
 issn = {0232-704X},
 volume = {41},
 number = {4},
 pages = {391--405},
 year = {2012},
 doi = {10.1007/s10455-011-9288-7},
}
\end{document}